\documentclass[reqno]{amsart}

\usepackage{xypic,hyperref,todonotes,comment,tikz}
\usetikzlibrary{arrows,calc,decorations.markings,math,arrows.meta}
\usepackage{graphics,amssymb,mathtools}
\usepackage{latexsym}
\usepackage{mathrsfs}
\xyoption{curve}
\usepackage{hyperref}
\hypersetup{ 
colorlinks,
citecolor=blue,
filecolor=blue,
linkcolor=black,
urlcolor=blue
}

\newtheorem{lemma}{Lemma}[section]
\newtheorem{proposition}[lemma]{Proposition}
\newtheorem{theorem}[lemma]{Theorem}
\newtheorem{corollary}[lemma]{Corollary}
\newtheorem*{conjecture}{Conjecture}
\newtheorem*{theoremA}{Theorem}

\theoremstyle{definition}
\newtheorem{example}[lemma]{Example}
\newtheorem{definition}[lemma]{Definition}
\newtheorem{remark}[lemma]{Remark}

\newcommand{\mc}[1]{\mathcal{{#1}}}
\newcommand{\mbb}[1]{\mathbb{#1}}
\newcommand{\mrm}[1]{\mathrm{#1}}
\newcommand{\mtt}[1]{\mathtt{#1}}

\newcommand{\Hom}{\mathrm{Hom}}
\newcommand{\End}{\mathrm{End}}

\newcommand{\Ext}{\mathrm{Ext}}
\newcommand{\Spec}{\mathrm{Spec}}
\newcommand{\ot}{\otimes}

\newcommand{\cA}{\mc{A}}

\newcommand{\cB}{\mc{B}}
\newcommand{\cC}{\mc{C}}
\newcommand{\GL}{\mathrm{GL}}

\newcommand{\D}{\mathrm{D}}
\newcommand{\gl}{\mathrm{gl}}
\newcommand{\g}{\mathtt{g}}

\newcommand{\G}{\mathrm{G}}

\newcommand{\Z}{\mathscr{Z}}
\newcommand{\HH}{H\! H}

\newcommand{\rep}{\mathtt{rep}}
\newcommand{\gr}{\mathrm{gr}}

\newcommand{\bG}{\mathbb{G}}
\newcommand{\e}{\varepsilon}
\renewcommand{\G}{\mathbb{G}}
\renewcommand{\u}{\mathfrak{u}}
\renewcommand{\hat}{\widehat}
\renewcommand{\O}{\mathscr{O}}
\renewcommand{\b}{\ast}

\title[Cohomology for doubles of infinitesimal group schemes]{Cohomology for Drinfeld doubles of some infinitesimal group schemes}
\date{\today}

\author{Eric Friedlander}
\author{Cris Negron}
\thanks{The second author was supported by NSF Postdoctoral Research Fellowship DMS-1503147}

\address{Eric Friedlander, Department of Mathematics, University of Southern California, Los Angeles, CA, USA}
\email{ericmf@usc.edu}

\address{Cris Negron, Department of Mathematics, Massachusetts Institute of Technology, Cambridge, MA, USA}
\email{negronc@mit.edu}

\begin{document}
\maketitle

\begin{abstract}
Consider a field $k$ of characteristic $p > 0$, $\G_{(r)}$ the $r$-th Frobenius kernel of a smooth algebraic group $\G$, $\D\G_{(r)}$ the Drinfeld double of $\G_{(r)}$, and $M$ a finite dimensional $\D\G_{(r)}$-module.  We prove that the cohomology algebra $H^*(\D\G_{(r)},k)$ is finitely generated and that $H^*(\D\G_{(r)},M)$ is a finitely generated module over this cohomology algebra.  We exhibit a finite map of algebras $\theta_r:H^*(\G_{(r)},k) \otimes S(\g) \to H^*(\D\G_{(r)},k)$ which offers an approach to support varieties for $\D\G_{(r)}$-modules.   For many examples of interest, $\theta_r$ is injective and induces an isomorphism of associated reduced schemes.  Additionally, for $M$ an irreducible $\D\G_{(r)}$-module, $\theta_r$ enables us to identify the support
variety of $M$ in terms of the support variety of $M$ viewed as a $\G_{(r)}$-module.
\end{abstract}

\section{Introduction}

For a Hopf algebra $A$ over a field $k$, we denote by $H^*(A,k)=\Ext^*_A(k,k)$ the Hopf
cohomology and we denote by $H^*(A,V) = \Ext_A^*(k,V)$ the cohomology of $A$ with values 
in a finite dimensional $A$-module $V$.  The goal of this paper is to prove the following conjecture for an interesting class of examples.

\begin{conjecture}[The finite generation conjecture]
For any finite dimensional Hopf algebra $A$, and finite dimensional $A$-module $V$, the cohomology $H^\b(A,k)$ is a finitely generated $k$-algebra and $H^\b(A,V)$ is a finitely generated module over $H^\b(A,k)$.
\end{conjecture}

The conjecture has existed as a question at least since the 90's (see e.g.~\cite{friedlandersuslin97}), and was recently 
stated explicitly in the work of Etingof and Ostrik~\cite{etingofostrik04}. 
In the finite characteristic setting, the conjecture was verified for cocommutative Hopf algebras in the work of Friedlander and Suslin~\cite{friedlandersuslin97} in the $90$'s.  This followed earlier work of Friedlander and Parshall on the cohomology of restricted 
enveloping algebras~\cite{friedlanderparshall86}.   More recently, Drupieski generalized these results to finite super groups~\cite{drupieski16} (i.e., cocommutative Hopf algebras in the symmetric category of $\mathbb{Z}/2\mathbb{Z}$-graded vector spaces).
\par

For a commutative Hopf algebra $A$ over a field of characteristic $p$, one can arrive at the desired finite generation result from the existence of an abstract {\it algebra} isomorphism $A\cong k[\mathbb{Z}/p^{l_1}\mathbb{Z}]\ot\dots\ot k[\mathbb{Z}/p^{l_n}\mathbb{Z}]$ whenever $A$ is local, and the fact that the cohomology $H^\b(A,V)$ only depends on the algebra structure of $A$.
\par

In characteristic $0$ most of the work to date has focused on pointed Hopf algebras.  In~\cite{ginzburgkumar93} (see also~\cite{bnpp14}) Ginzburg and Kumar showed that small quantum groups have finitely generated cohomology.  In~\cite{mpsw10}, Mastnak, Pevtsova, Schauenburg, and Witherspoon verified the finite generation conjecture for most pointed Hopf algebras with abelian group of grouplikes.  Such Hopf algebras were classified by Andruskiewitsch and Schneider, and can be understood broadly as deformations of small quantum groups~\cite{AS}.  For results concerning pointed Hopf algebras with non-abelian grouplikes one can see~\cite{stefanvay16}.
\par

In this work we consider Drinfeld doubles of finite group schemes in characteristic $p>0$.  We recall that the Drinfeld double $\D G$ of a finite group scheme $G$ is the smash product 
\[
\D G=\O(G)\# kG
\]
of the group algebra $kG$ of $G$ acting via the adjoint action on the algebra $\O(G)=(kG)^\ast$ of functions on $G$.  The coalgebra structure on $\D G$ is the product structure $\O(G)^\mrm{cop}\ot kG$. The Drinfeld double $\D G$ is neither commutative nor cocommutative (unless $G$ is commutative) and rarely pointed.  For some examples of the computational and theoretical significance of the double one can see~\cite{etingofgelaki02,etingof02,ksz06,ngschauenburg07,shimizu17}.
\par

Our finite generation results for Drinfeld doubles apply to other Hopf algebras thanks to various general properties of the Drinfeld
double construction.  For example, $\rep(\D A)\cong \rep(\D (A^\ast))$; and for any cocycle twist $\sigma$ (see~\cite{montgomery04}),
$\rep(\D A)\cong \rep(\D (A_\sigma))$~\cite{majidoeckl99,benkartetal10}.
\par

Let us now fix $k$ a field of finite characteristic $p$.  We assume additionally that $p$ is odd, although most of our results will still hold when $p=2$ (see Section~\ref{sect:2}).  Recall that the $r$-th Frobenius kernel $\G_{(r)}$ is the group scheme theoretic kernel of the $r$-th Frobenius map $F^r:\G\to \G^{(r)}$ (see Section~\ref{sect:frob}).  We refer the reader to \cite{sullivan78,cline87,jantzen07} for some discussion of the important role Frobenius kernels play in the modular representation theory of algebraic groups.

We prove the following:

\begin{theoremA}[\ref{thm:fingen},\ \ref{thm:fingen2}]
Consider the $r$-th Frobenius kernel $\G_{(r)}$ of a smooth algebraic group $\G$.  The cohomology of the double $H^\b(\D\G_{(r)},k)$ 
is a finitely generated $k$ algebra.  Moreover, for any finite dimensional $\D\G_{(r)}$-module $M$, the cohomology $H^\b(\D\G_{(r)},M)$ 
is a finitely generated $H^\b(\D\G_{(r)},k)$-module.
\end{theoremA}

Our approach utilizes associations between deformation theory and Hopf cohomology.  We show that the deformation 
$\G_{(r+1)}$ of $\G_{(r)}$ produces a natural map  $\sigma_\O:\g^{(r)}\to H^2(\O(\G_{(r)}),k)$, where $\g=\mathrm{Lie}(\G)$.  
The map $\sigma_\O$ has a natural lift to the cohomology of the double $\sigma_\D:\g^{(r)}\to H^2(\D\G_{(r)},k)$, which is 
again constructed in a deformation theoretic manner.  
The smoothness hypothesis of $\G$ plays an important role in our proof.  Namely, we require an inductive argument passing 
from a finite group scheme of height $r$ to one of height $r+1$ which uses in an essential way the structure of $\bG_{(r+1)}$ as a flat extension of $\bG_{(r+1)}/\bG_{(r)}$ to obtain cohomology classes via deformation theory.
\par

 In proving the above theorem, we construct a finite algebra map
\[
\theta_r:H^\ast(\G_{(r)},k)\ot S(\g^{(r)}[2])\to H^\ast(\D\G_{(r)},k)
\]
(using $\sigma_\D$ and the inflation $H^\ast(\G_{(r)},k)\to H^\ast(\D\G_{(r)},k)$) with associated map of reduced affine schemes
\begin{equation}\label{eq:1}
\Theta_r: |\D\G_{(r)}|\to |k\G_{(r)}|\times (\g^\ast)^{(r)}
\end{equation}
(see Theorem~\ref{thm:fingen}).   Here, and elsewhere, we employ the usual notation $S(V[n])$ for the the symmetric algebra over $k$ of 
the $k$-vector space $V$ placed in degree $n$, and we use the notation $|A|=\Spec\!\ H^{ev}(A,k)_{\mathrm{red}}$ for the reduced spectrum of the cohomology of a Hopf algebra $A$. 
\par

For many classical algebraic groups $\G$ we are able to deduce additional information concerning $\Theta_r$ as formulated in
the following theorem. 

\begin{theoremA}[\ref{cor:|class|}]
Let $\G$ be a general linear group, simple algebraic group, Borel subgroup in a simple algebraic group, or a unipotent subgroup in a semisimple algebraic group which is normalized by a maximal torus.  Suppose that $p$ is very good for $\G$, or that $p>cl(\G)$ in the unipotent case.
\begin{itemize}
\item If $p>\dim\G+1$ then the map $\Theta_r$ of ~\eqref{eq:1} is an isomorphism for all $r$.
\item For arbitrary $p$, the map $\Theta_r$ is an isomorphism whenever $r$ is such that $p^r>2\dim\G$.
\end{itemize}
\end{theoremA}

The key observation we use in proving the above theorem is that the hypotheses guarantee the existence of a quasi-logarithm $L: \G \to \g$
\cite{kazhdanvarshavsky06}.
This leads to a grading on the Drinfeld double $\D\G_{(r)}$ which much simplifies the analysis of the LHS spectral sequence 
we use to investigate the cohomology of $\mrm{D}\G_{(r)}$.
The ``very good" condition on $p$ is a mild condition which we review in Section~\ref{sect:qlog}.  In the unipotent case, the integer $cl(\G)$ is the nilpotence class of $\G$, which is always less than $\dim(\G)$.  The theorem implies an equality of dimensions $\dim|\D\G_{(r)}|=\dim|k\G_{(r)}|+\dim\G$ for such classical groups.
\par

We also consider the support variety $|\D\G_{(r)}|_M$ associated to a $\D\G_{(r)}$-module $M$.  The support variety for $M$ is defined as the closed, reduced, subscheme in $|\D\G_{(r)}|$ defined by the kernel of the algebra map 
\[
-\ot M:H^{ev}(\D\G_{(r)},k)\to \Ext^{ev}_{\D\G_{(r)}}(M,M).
\]

\begin{theoremA}[\ref{cor:|class|V}]
Suppose $\G$ is as in the statement of the previous theorem.  If $p>\dim\G+1$ or $p^r>2\dim\G$, 
then for any  irreducible $\D\G_{(r)}$-module $M$  the map $\Theta_r$ of ~\eqref{eq:1} restricts to an isomorphism of schemes
\[
\Theta_{r,M}: |\D\G_{(r)}|_M\ \stackrel{\sim}{\to} \ |k\G_{(r)}|_M\times (\g^\ast)^{(r)}.
\]
\end{theoremA}

We supplement the preceding results by extending many of them to relative Drinfeld
doubles (see Section~\ref{sect:relative}).

\subsection*{Organization}

In Section~\ref{sect:defo}, we discuss associations between deformations and Hopf cohomology, and produce the aforementioned 
maps $\sigma_\O$ and $\sigma_\D$.  In Section~\ref{sect:proof} we prove that the algebra map 
$S(\g^{(r)}[2])\to H^\ast(\O(\G_{(r)}),k)_{\mathrm{red}}$ induced by $\sigma_\O$ is an isomorphism.  
We use the lifting $\sigma_\D$, in conjunction with the inflation map $H^\ast(\G_{(r)},k)\to H^\ast(\D\G_{(r)},k)$, 
to establish the finite generation of cohomology for the double $D(\G_{(r)})$ in Section~\ref{sect:fg}.  Section~\ref{sect:qlog} is 
dedicated to an analysis of classical groups at large primes.  Section~\ref{sect:supp} is dedicated to support varieties.

\subsection*{Acknowledgments}

Thanks to Roman Bezrukavnikov, Robert Guralnick, Julia Pevtsova, Julia Plavnik, and Sarah Witherspoon for helpful conversations.  We are particularly grateful to the referee for detailed, constructive comments.

\section{Finite group schemes and the Frobenius}
\label{sect:frob}

We fix from this point on $k$ a field of finite characteristic $p$.  We assume $p\neq 2$ (see Section~\ref{sect:2}).  A ``scheme" is a scheme of finite type over $k$, and $\ot=\ot_k$.  All schemes considered in this work will be affine.
\par

For an affine group scheme $G$, a rational (left) $G$-representation is a 
(right) comodule over the coordinate algebra $\O(G)$.  A $G$-algebra is a $\O(G)$-comodule algebra (i.e., an algebra $R$ which is a rational $G$-representation in such a way that the multiplication $R\ot R\to R$ is a map of $G$-representations).  We let $H^\ast(G,M)$ denote the rational group cohomology of $G$ with coefficients in $M$.  If $G$ is a finite group scheme with Hopf algebra $kG$ (the ``group algebra"
of $G$), then $H^\ast(G,M) = H^\ast(kG,M)$.
\par

In this section we review some standard information on Frobenius maps and Frobenius kernels.  One can see Jantzen's book~\cite{jantzen07} for a more detailed presentation.

\subsection{Frobenius maps and Frobenius kernels}

Let $\phi^r:k \to k$ be the  $p^r$-th power map on $k$, $\lambda \mapsto \lambda^{p^r}$.
\par

Given an affine group scheme $G$ we let $G^{(r)}$ denote the fiber product of $G$ with $\Spec(k)$ along $\phi^r$,
\[
\xymatrix{
G^{(r)}\ar[r]\ar[d] & G\ar[d]\\
\Spec(k)\ar[r]^{(\phi^r)^\ast} & \Spec(k).
}
\]
By functoriality of the pullback we see that $(-)^{(r)}$ provides a functor on the category of group schemes over $k$.  There is a natural map of group schemes $F^r:G\to G^{(r)}$ given explicitly as follows: $\O(G^{(r)}) = \O(G) \otimes_{\phi^r} k \to \O(G)$ sends $f \otimes_{\phi^r} 
\lambda$ to $\lambda \cdot f^{p^r}$.

\begin{definition}
\begin{enumerate}
\item[i)] The above map $F^r:G\to G^{(r)}$ is called the $r$-th (relative) Frobenius map.
\item[ii)] The $r$-th Frobenius kernel $G_{(r)}$ of $G$ is the group scheme theoretic kernel of the $r$-th Frobenius map, $1\to G_{(r)}\to G\overset{F^r}\to G^{(r)}$.
\item[iii)] We say $G$ is of height $\leq r$ if $G=G_{(r)}$.
\end{enumerate}
\end{definition}

The closed subgroup scheme $G_{(r)}$ in $G$ is the spectrum of the quotient Hopf algebra
\[
\O(G_{(r)})= \O(G)/(f^{p^r}:f\in m_G),
\]
where $m_G$ is the maximal ideal corresponding to the identity in $G$.  Whence we see that an affine group scheme $G$ is of height $\leq r$ if and only if $f^{p^r}=0$ for each $f\in m_G$.

\begin{example}
For $G$ a height $1$ group scheme, we have $G=\Spec(\u(\g)^\ast)$ where $\g$ is the restricted Lie algebra for $G$ and $\u(\g)$ is the restricted enveloping algebra.  This association gives a natural bijection between height $1$ group schemes and finite dimensional restricted Lie algebras.
\end{example}
\begin{example}
Consider $\GL_n$.  This is the spectrum of the Hopf algebra 
\[
\O(\GL_n)=k[x_{ij},\ \! {\det}^{-1}:1\leq i,j\leq n]
\]
with comultiplication $\Delta(x_{ij})=\sum_{k=1}^n x_{ik}\ot x_{kj}$, counit $\epsilon(x_{ij})=\delta_{ij}$, and antipode given by the adjoint formula for the inverse of a matrix.  The Frobenius kernels in this case are given by
\[
\O(\GL_{n(r)})=k[x_{ij}:1\leq i,j\leq n]/(x_{ij}^{p^r}-\delta_{ij}).
\]
Note that in the above presentation of the Frobenius kernel the determinant is already invertible.
\end{example}

\subsection{Frobenius twists of representations}
\label{sect:frob2}

For a rational $G$-representation $V$ we let $V^{(r)}$ denote the new $G$-representation which is the vector space $k\ot_{\phi^r} V$ along with the $G$-action given by the composite
\[
G\overset{F^r}\longrightarrow G^{(r)}\to \GL(V)^{(r)}=\GL(V^{(r)}).
\]
The tensor product $k\ot_{\phi^r}$ here denotes base change along $\phi^r$.  As a comodule, $V^{(r)}$ has right $\O(G)$-coaction given by
\[
\rho^{(r)}(c\ot v)=\sum_i (c\ot v_{i_0})\ot v_{i_1}^{p^r},
\]
where the initial coaction of $\O(G)$ on $V$ is given by $\rho(v)=\sum_i v_{i_0}\ot v_{i_1}$.  (In the above equation $c\in k$ and $v\in V$.)  We call $V^{(r)}$ the $r$-th Frobenius twist of $V$.  The proof of the following lemma is immediate from the observation that the composition $\O(G^{(r)}) \to \O(G) \to \O(G_{(r)})$ factors through the counit for $\O(G^{(r)})$.

\begin{lemma}
For $G$ of height $\leq r$, and $V$ any rational $G$-representation, $G$ acts trivially on the $r$-th Frobenius twist $V^{(r)}$.
\end{lemma}

We also employ a natural isomorphism of $G$-representations $(V^\ast)^{(r)}\overset{\cong}\to (V^{(r)})^\ast$ given by the formula $c\ot f\mapsto (c'\ot v\mapsto cc'f(v)^{p^r})$.

\section{Deformations of Frobenius kernels and cohomology}
\label{sect:defo}

We fix a positive integer $r$ and consider the $r$-th Frobenius kernel $\bG_{(r)}$ of
a smooth linear algebraic group over $k$, a field of odd characteristic $p > 0$.  We denote by $\D\bG_{(r)}$ the Drinfeld double of 
the Hopf algebra $k\G_{(r)}$.  This is the smash product $\D\G_{(r)}=\O(\G_{(r)})\# k\G_{(r)}$ of the coordinate algebra $\O(\G_{(r)})$ with the group algebra $k\G_{(r)}$ with respect to
the right adjoint action of $\G_{(r)}$ on itself~\cite[Cor.\ 10.3.10]{montgomery93}.
\par

The adjoint action of $\G_{(r)}$ on itself corresponds specifically to the $\O(\G_{(r)})$-coaction $\rho(f)=\sum_if_{i_2}\ot S(f_{i_1})f_{i_3}$, and subsequent $k\G_{(r)}$-action $\xi\cdot f=\sum_if_{i_2}\xi(S(f_{i_1})f_{i_3})$.  The Hopf structure on $\D\G_{(r)}$ is the unique one so that the two inclusions $\O(\G_{(r)})^\mrm{cop}\to \D\G_{(r)}$ and $k\G_{(r)}\to \D\G_{(r)}$ are maps of Hopf algebras.
\par

We proceed to construct cohomology classes in $H^2(\D\G_{(r)},k)$ which will enable our proof of finite generation in Section~\ref{sect:fg}.  Our construction involves deformations of $\O(\G_{(r)})$ and $\D\bG_{(r)}$, in particular the embedding $\G_{(r)} \to \G_{(r+1)}$ which we view as a deformation of $\G_{(r)}$ parametrized by $\G_{(r+1)}/\G_{(r)}$.  This deformation leads to classes in the Hochschild cohomology group $\HH^2(\D\G_{(r)},k)$ 
and thereby classes in $H^2(\D\G_{(r)},k)$.

\subsection{Hochschild cohomology and deformations}

We recall that the Hochschild cohomology of an algebra $R$ with coefficients in an $R$-bimodule $M$ is defined as
\[
\HH^*(R,M) \ \equiv \ \Ext_{R\ot R^{op}}^*(R,M),
\]
and $\HH^\ast(R)=\HH^\ast(R,R)$.  Thus, $\HH^*(R,M)$ is functorial with respect to maps $M \to N$
of $R$-bimodules.  Moreover, we have the well-known surjection (see \cite[Sect.\ 5.6]{ginzburgkumar93} or~\cite[Sect.\ 7]{pevtsovawitherspoon09})
$$\HH^*(R) \ \cong H^*(R,R^{ad}) \twoheadrightarrow H^*(R,k) \equiv \Ext_R^*(k,k), \quad \text{if } \ R \ \text{is a Hopf algebra}$$ 
(using the fact that $k \to R^{ad}$ splits). 
We further recall that a (infinitesimal) deformation $\mathscr{R}$ of an algebra $R$ parametrized by a scheme $\Spec(A)$ is a flat $A$-algebra, 
where $A$ is an Artin local (commutative) algebra with residue field $k$, equipped with a fixed isomorphism 
$\mathscr{R}\ot_A k\overset{\cong}\to R$.
Given any map $A\to A'$ of such Artinian local algebras, and a deformation $\mathscr{R}$ parametrized by $\Spec(A)$, 
we can extend $\mathscr{R}$ along $A\to A'$ to get a deformation $\mathscr{R}\ot_A A'$ parametrized by $\Spec(A')$.  
Two deformations $\mathscr{R}$ and $\mathscr{R}'$ parametrized by $\Spec(A)$ are said to be isomorphic 
if there is an $A$-algebra isomorphism $l:\mathscr{R}\to \mathscr{R}'$ fitting into a diagram
\[
\xymatrix{
\mathscr{R}\ar[rr]^l\ar[dr] & & \mathscr{R}'\ar[dl]\\
  & R &.
}
\]
A special role is played by deformations parametrized by $\Spec(k[\varepsilon])$, where $k[\varepsilon] \equiv k[t]/t^2$ is the Artin local
algebra of ``dual numbers".

\begin{theorem}[{Gerstenhaber~\cite{gerstenhaber64}}]
\label{thm:Gerstenhaber}
There is a naturally constructed bijection
\begin{equation}\label{eq:bij}
\frac{\{\mathrm{deformations\ of\ }R\mathrm{\ parametrized\ by\ }\Spec(k[\varepsilon])\}}{\cong} \
\stackrel{\sim}{\longrightarrow} \ \HH^2(R).
\end{equation}
\end{theorem}

The domain of the above bijection has a natural linear structure under which~\eqref{eq:bij} is a linear isomorphism.  Let us explain some of the details of Gerstenhaber's result.
\par

Consider a deformation $\mathscr{R}$ of $R$ over $\Spec(k[\varepsilon])$.  By choosing a $k[\e]$-linear isomorphism $R[\e]\equiv R\ot k[\varepsilon]\cong \mathscr{R}$ the deformation $\mathscr{R}$ may be identified 
with the $k[\varepsilon]$-module $R[\e]$ equipped with a multiplication
\[
a\cdot_{\mathscr{R}}b=ab+F_{\mathscr{R}}(a,b)\varepsilon, \quad a,b\in R\subset R\ot k[\varepsilon].
\]
The function $F_{\mathscr{R}}:R\ot R\to R$ defines a $2$-cocycle in the standard Hochschild cochain complex
\[
C^\b(R)=0\to R\to \Hom_k(R,R)\to \Hom_k(R\ot R,R)\to \Hom_k(R^{\ot 3},R)\to\cdots.
\]
This determines a map from deformations to $\HH^2(R)$.  
To define the inverse map, one simply uses a $2$-cocycle in the standard Hochschild cochain complex to 
define a mulplication on $R[\varepsilon]$.
The addition of isoclasses of deformations $[\mathscr{R}]+[\mathscr{R}']$ corresponds to addition of the 
functions $F_{\mathscr{R}}+F_{\mathscr{R}'}$, and scaling $c[\mathscr{R}]$ corresponds to scaling the function $c F_{\mathscr{R}}$.

The following lemma should be standard.

\begin{lemma}\label{lem:sigma_R}
Let $\mathscr{R}$ be an (infinitesimal) deformation of $R$ parametrized by $S=\Spec(A)$.  Then there is a $k$-linear mapping
\[
\Sigma_\mathscr{R}:T_pS\to \HH^2(R)
\]
which sends an element $\xi\in T_pS=\Hom_{\mathrm{Alg}}(A,k[\e])$ to the class corresponding to the deformation $\mathscr{R}\ot_Ak[\e]$, where we change base via $\xi$.
\end{lemma}

In the statement of the above lemma $p$ is the unique point in $S$.

\begin{proof}
Given $\xi\in T_pS$ we let $Def_\xi=\mathscr{R}\ot_Ak[\e]$ denote the corresponding deformation.  For the proof we identify the tangent space $T_pS$ with $k$-linear maps $m_A\to k$ which vanish on $m_A^2$, where $m_A$ is the unique maximal ideal of $A$.  
We adopt an $A$-linear identification $\mathscr{R}=R\ot A$, and write the multiplication on $\mathscr{R}$ as $r\cdot_\mathscr{R} r'=rr'+E(r,r')$, where $r,r'\in R$ and $E$ is a linear function $E:R\otimes R\to R\ot m_A$.
\par

If we take $F_\xi=(1\ot \xi)E$, for $\xi\in T_pS$, then the multiplication on the base change $Def_\xi$ is given by $r\cdot_\xi r'=rr'+F_\xi(r,r')\e$.  Whence we have an equality in Hochschild cohomology
\[
\Sigma_{\mathscr{R}}(\xi)=\big[Def_\xi\big]=[F_\xi]\in \HH^2(R).
\]
By the definition of $F_\xi$ we see that $F_{c\xi+c'\xi'}=cF_\xi+c'F_{\xi'}$.  It follows that the map $\Sigma_{\mathscr{R}}:T_pS\to \HH^2(R)$ is $k$-linear.
\end{proof}

\begin{definition}
Given a deformation $\mathscr{R}$ of a Hopf algebra $R$ parametrized by $S$, we let
\[
\sigma_{\mathscr{R}}:T_pS\to H^2(R,k)
\]
denote the composite $T_pS\overset{\Sigma_\mathscr{R}}\longrightarrow \HH^2(R)\to H^2(R,k)$, where $\Sigma_{\mathscr{R}}$ is as in Lemma~\ref{lem:sigma_R}.
\end{definition}

\subsection{Cohomology classes for the coordinate algebra via deformations}
\label{sect:HO}

For the remainder of this section, we fix $\G$ a smooth (affine) algebraic group of dimension $n$ and a positive integer $r$.  We take $\g = \mathrm{Lie}(\G) = \mathrm{Lie}(\G_{(s)})$ for any $s \geq 1$; in particular, $\g = \mathrm{Lie}(\G_{(r)})$.  
We shall view $\O(\bG_{(r+1)})$ as a deformation of $\O(\G_{(r)})$ parametrized by $\G_{(r+1)}/\G_{(r)}$.  
One sees this geometrically using the pull-back square
\begin{equation}
\label{Frobquot}
\xymatrix{
\G_{(r)} \ar[r] \ar[d] & \G_{(r+1)} \ar[d] \\
\Spec(k) \ar[r] & \G_{(r+1)}/\G_{(r)}
}
\end{equation}

\begin{proposition}
The extension $\O(\G_{(r+1)}/\G_{(r)}) \to  \O(\bG_{(r+1)})$ is a deformation of $\O(\G_{(r)})$ parametized by 
$\G_{(r+1)}/\G_{(r)} \cong \G_{(1)}^{(r)}$.  We refer to this deformation of $\O(\G_{(r)})$ as $\O_{\mathrm{nat}}$.
\end{proposition}

\begin{proof}
The isomorphism $\G_{(r+1)}/\G_{(r)} \cong \G_{(1)}^{(r)}$ is induced by the Frobenius $\G_{(r+1)}\to \G^{(r)}$, and can be found in~\cite[Prop.\ I.9.5]{jantzen07}.  The fact that $\O(\G_{(r+1)}/\G_{(r)}) \to \O(\bG_{(r+1)})$ is a deformation of $\O(\G_{(r)})$ follows easily from the diagram~\eqref{Frobquot}.
\end{proof}

Take $\O=\O(\G_{(r)})$.  Note that $\g^{(r)}=T_1\G^{(r)}_{(1)}$.  We get from Lemma~\ref{lem:sigma_R} and $\O_{\mathrm{nat}}$ a canonical linear map $\sigma_{\O}=\sigma_{\O_{\mathrm{nat}}}:\g^{(r)}\to H^2(\O,k)$ and induced algebra map 
\[
\sigma'_\O:S(\g^{(r)}[2])\to H^\ast(\O,k),
\]
where $S(-)$ denotes the symmetric algebra.  From the identification $H^1(\O,k)=T_1\G_{(r)}=\g$, in conjunction with $\sigma'_\O$, we get yet another algebra map
\begin{equation}\label{eq:374}
\wedge^\ast(\g[1])\ot S(\g^{(r)}[2])\to H^\ast(\O,k).
\end{equation}
In Section~\ref{sect:proof} below we will prove the following proposition.

\begin{proposition}\label{prop:381}
The algebra map~\eqref{eq:374} is an isomorphism of $\G_{(r)}$-algebras.  In particular, $\sigma'_\O:S(\g^{(r)}[2])\to H^\ast(\O(\G_{(r)}),k)$ is an isomorphism modulo nilpotents.
\end{proposition}

The $\mathbb{G}_{(r)}$-action on the product $\wedge^\ast(\g[1])\ot S(\g^{(r)}[2])$ is induced by the adjoint action on $\g$ and the trivial action on its twist $\g^{(r)}$.

\begin{remark}
We can easily establish an {\it abstract} algebra isomorphism between $\wedge^\ast(\g)\ot S(\g^{(r)}[2])$ and the cohomology $H^\ast(\O,k)$
as follows.  As verified in \cite[Thm.\ 14.4]{waterhouse12}, the fact that $\bG_{(r)}$ is connected implies that there is an isomorphism $\O \cong
k[x_1,\ldots,x_n]/(x_1^{p^{e_1}},\ldots,x_n^{p^{e_n}})$ for some $n, e_1,\ldots,e_n > 0$.  The well known computation of $H^\ast(k[x]/(x^{p^e}),k)
\simeq H^\ast(\mathbb Z/p^e,k)$
and the K\"unneth Theorem thus implies the asserted isomorphism.
The significance of Proposition~\ref{prop:381} is that we may use the deformation map $\sigma_\O$ to arrive at such an isomorphism.  We will see below that $\sigma_\O$ admits a lift to the cohomology of the double $\D\G_{(r)}$.  The existence of such a lift is an essential point in the proof that the cohomology of the double is finitely generated.
\end{remark}

\subsection{Cohomology classes for the double via deformations}

Since $\G_{(r)}$ acts trivially on the quotient $\G_{(r+1)}/\G_{(r)}$ we see that the image of the inclusion
\[
\O(\G_{(r+1)}/\G_{(r)})\to \O(\G_{(r+1)})=\O_{\mathrm{nat}}
\]
is in the $\G_{(r)}$-invariants.  Hence the induced inclusion into the smash product
\[
\O(\G_{(r+1)}/\G_{(r)})\to \O_{\mathrm{nat}}\# k\G_{(r)}
\]
has central image, where $\G_{(r)}$ acts via the adjoint action on $\O_{\mathrm{nat}}$.  Furthermore, the reduction $(\O_{\mathrm{nat}}\# k\G_{(r)})\ot_{\O(\G_{(r+1)}/\G_{(r)})}k$ recovers the double $\D\G_{(r)}$.  Whence we have that the above smash product is a deformation of the double parametrized by $\G^{(r)}_{(1)}\cong \G_{(r+1)}/\G_{(r)}$.  We denote this deformation $\D_{\mathrm{nat}}=\O_{\mathrm{nat}}\# k\G_{(r)}$.
\par

The deformation $\D_{\mathrm{nat}}$ induces a map to cohomology
\[
\sigma_\D\equiv\sigma_{\D_{\mathrm{nat}}}:\g^{(r)}\to H^2(\D\G_{(r)},k)
\]
and subsequent graded algebra morphism $\sigma'_\D:S(\g^{(r)}[2])\to H^\ast(\D\G_{(r)},k)$.  

\begin{proposition}\label{prop:g2HD}
The triangle
\begin{equation}\label{eq:389}
\xymatrixrowsep{4mm}
\xymatrix{
 & H^2(\D\G_{(r)},k)\ar[dr]^{\mathrm{res}} &  \\
\g^{(r)}\ar[ur]^{\sigma_D}\ar[rr]^{\sigma_\O} & & H^2(\O(\G_{(r)}),k)
}
\end{equation}
commutes.
\end{proposition}

\begin{proof}
Take $\O=\O(\G_{(r)})$ and $\D=\D\G_{(r)}$.  The diagram~\eqref{eq:389} follows from the diagram
\[
\xymatrix{
\O_{\mathrm{nat}}\ar[rr]^{incl}\ar[d] & & \D_{\mathrm{nat}}\ar[d]\\
\O\ar[rr]^{incl} & & \D,
}
\]
where the top map is one of $\O(\G_{(r+1)}/\G_{(r)})$-algebras and the vertical maps are given by applying 
$(-) \ot_{\O(\G_{(r+1)}/\G_{(r)})}k$.  
In particular, the commutative square implies that the maps $E^D$ and $E^\O$ from the proof of Lemma~\ref{lem:sigma_R} 
can be chosen in a compatible manner so that $E^D|_{\O\ot \O}=E^\O$.  Hence the resulting Hopf 2-cocycles 
$\bar{F}^\D_\xi$ and $\bar{F}^\O_\xi$, corresponding to an element $\xi\in \g^{(r)}$, are such that
\[
\mathrm{res}(\sigma_\D(\xi))=\mathrm{res}([\bar{F}^\D_\xi])=[\bar{F}^\D_\xi|_{\O\ot\O}]=[\bar{F}^\O_\xi]=\sigma_\O(\xi).
\]
\end{proof}

\begin{corollary}\label{cor:g2inv}
The map $\sigma_\O:\g^{(r)}\to H^2(\O(\G_{(r)}),k)$ of Section~\ref{sect:HO} has image in the invariants $H^2(\O(\G_{(r)}),k)^{\G_{(r)}}$.
\end{corollary}

\begin{proof}
The restriction $H^\b(\D,k)\to H^\b(\O,k)$ is induced by the cochain inclusion
\[
\Hom^\b_\D(P,k)=\Hom^\b_\O(P,k)^{\G_{(r)}}\to \Hom^\b_\O(P,k),
\]
where $P$ is any resolution of $k$ over $\D$.  Hence the lifting of Proposition~\ref{prop:g2HD} implies that $\sigma_{\O}$ has image in the $\G_{(r)}$-invariants.
\end{proof}

We can consider also the inflation $H^\ast(\G_{(r)},k)\to H^\ast(\D\G_{(r)},k)$ induced by the Hopf projection $\D\G_{(r)}\to k\G_{(r)}$.  This inflation, in conjunction with the algebra map $\sigma'_\D$, represent contributions to the cohomology of the double coming from the two constituent factors $k\G_{(r)}$ and $\O$.

\begin{definition}\label{def:g2HD}
We let
\[
\theta_r:H^\ast(\G_{(r)},k)\ot S(\g^{(r)}[2])\to H^\ast(\D\G_{(r)},k)
\]
denote the product of the inflation from $H^\ast(\G_{(r)},k)$ and $\sigma'_\D$.
\end{definition}

We will find in Section~\ref{sect:fg} that the map $\theta_r$ is finite.  It will follow that the cohomology of the double is finitely generated.


\section{Proof of Proposition~\ref{prop:381}}
\label{sect:proof}

For a deformation $\mathscr{R}$ of an algebra $R$ parametrized by $S=\Spec(A)$, we view the tangent space $T_pS$ as the first cohomology $H^1(A,k)$.  (Both of which are identified with algebra maps to the dual numbers $\Hom_{\mathrm{Alg}}(A,k[\e])$.)  So $\sigma_\mathscr{R}$ will appear as
\[
\sigma_\mathscr{R}:H^1(A,k)\to H^2(R,k).
\]
\par

In the case $\G=\G_a$, we will see that the map $\sigma_\O$ induced by the deformation $\O_{\mathrm{nat}}$ (which we denote by $\Z$ in this case)  behaves like the Bockstein map for the integral cohomology of a cyclic group with coefficients in $\mathbb{F}_p$.  In particular, it picks out an algebra generator in second cohomology.  From this observation we will deduce Proposition~\ref{prop:381} for general smooth $\G$.

\subsection{Generalized (higher) Bocksteins for $\G_{a}$}

In this subsection, we consider the special case $\G = \G_a$, the additive group (whose coordinate algebra is a polynomial 
algebra on one variable).
Consider the deformation $\O(\G_{a(r+1)})=k[t]/(t^{p^{r+1}})$ of $\O(\G_{a(r)})=k[t]/(t^{p^r})$ parametrized by $\O(\G_{a(1)}^{(r)})=(k[t]/(t^p))\ot_{\phi^r}k$.  To ease notation take $\Z=\O(\G_{a(r+1)})$, $Z=\O(\G_{a(r)})$ and $Z'=\O(\G_{a(1)}^{(r)})$.  The deformation $\Z$ produces a map
\[
\sigma_\Z:H^1(Z',k)\to H^2(Z,k).
\]
We let $\alpha\in H^1(Z',k)=\Hom_{\mathrm{Alg}}(Z',k[\e])$ denote the class given by the projection $Z'\to k[\e]$, $t\mapsto \e$.

\begin{definition}
Take $\beta\equiv\sigma_\Z(\alpha) \in H^2(Z,k)$.  We say that $\beta$ is the (higher order) Bockstein of the class $\alpha \in H^1(Z',k)$.
\end{definition}

Recall that for $i \geq0$ and $q> 1$ the cohomology $H^i(k[t]/(t^q),k)$ is $1$ dimensional.  (One can see this directly from the minimal, periodic, resolution of $k$.)  Hence $H^1(Z',k)$ and $H^2(Z,k)$ are $1$ dimensional.

\begin{lemma}\label{lem:349}
The map $\sigma_\Z:H^1(Z',k)\to H^2(Z,k)$ is a linear isomorphism.  In particular, $\beta$ is nonzero.
\end{lemma}

\begin{proof}
It suffices to show that the image $\beta$ of $\alpha\in H^1(Z',k)$ is nonzero.  Consider the base change $\Z\ot_{Z'}k[\e]$ via $\alpha$, and the $k[\e]$-linear identification $\Z\ot_{Z'}k[\e]\cong Z[\e]$ given by
\[
Z[\e]\overset{\cong}\longrightarrow \Z\ot_{Z'}k[\e],\ \ t^i\mapsto t^i\ot 1,\ t^i\varepsilon\mapsto t^i\ot \e.
\]
This induces a multiplication $z\cdot_\alpha z'=zz'+F_\alpha(z,z')\varepsilon$ on $Z[\e]$, where $F_\alpha$ is a Hochschild $2$-cocycle.  We have then $\left[\Z\ot_{Z'}k[\e]\right]=[F_\alpha]\in \HH^2(Z)$, and the corresponding Hopf cohomology class is $[\bar{F}_\alpha]\in H^2(Z,k)$, where $\bar{F}_\alpha$ is the composite of $F_\alpha$ with the counit $\bar{F}_\alpha=\epsilon F_\alpha$.
\par

We want to show that $\beta=\sigma_\Z(\alpha)=[\bar{F}_\alpha]$ is nonzero (i.e., that $\bar{F}_\alpha$ is not a coboundary).  One sees directly that $\bar{F}_\alpha(t^l,t^m)=\delta_{l+m,p^r}$, and in particular $\bar{F}_\alpha(t^i,t^{p^r-i})=1$.  One also sees that the differential of any degree $1$ function $f\in \Hom_k(Z,k)$ in the Hopf cochain complex for $Z$ is such that
\[
d(f)(t^i,t^{p^r-i})=\pm f(t^{p^r})=\pm f(0)=0.
\]
Therefore $\bar{F}_\alpha$ cannot be a coboundary, and the cohomology class $\beta=[\bar{F}_\alpha]$ is nonzero.
\end{proof}

We can consider now the $n$-th tensor product $\Z^{\ot n}$ as a deformation of $Z^{\ot n}$, parametrized by $\Spec\left((Z')^{\ot n}\right)$.  We let $\g_a$ denote the Lie algebra of $\G_{a}$ so that 
\[
(\g_a^{(r)})^n=H^1((Z')^{\ot n},k)=\Hom_{\mathrm{Alg}}((Z')^{\ot n},k[\e]),
\]
with each element $\sum_{i=1}^nc_i\alpha_i\in (\g_a^{(r)})^n$ corresponding to the algebra map
\[
\sum_ic_i\alpha_i:(Z')^{\ot n}\to k[\e],\ \ t_i\mapsto c_i\e.
\]
Here $\alpha_i$ is the basis vector for the $i$-th copy of $\g_a^{(r)}$, defined as above, and $t_i$ is the generator of the $i$-th factor in $(Z')^{\ot n}$.

\begin{proposition}\label{prop:439}
The map $\sigma_{\Z^{\ot n}}:(\g_a^{(r)})^n\to H^2(Z^{\ot n},k)$ induces an injective graded $k$-algebra map
\[
\sigma'_{\Z^{\ot n}}:S\left((\g_a^{(r)})^n[2]\right) \to H^*(Z^{\ot n} ,k)
\]
which is an isomorphism modulo nilpotents.
\end{proposition}

\begin{proof}
We claim that the reduction $\sigma_{\mathrm{red}}:(\g_a^{(r)})^n\to H^2(Z^{\ot n},k)_{\mathrm{red}}$ is injective.  (Here by $H^2(Z^{\ot n},k)_{\mathrm{red}}$ we mean the degree $2$ portion of the reduced algebra, and by $\sigma_\mrm{red}$ we mean the composite of $\sigma_{\Z^{\ot n}}$ with the reduction.)  It suffices to show that for any nonzero $\underline{c}=\sum_ic_i\alpha_i$ there is an index $j$ such that restriction along the factor $Z_j\to Z^{\ot n}$ produces a nonzero element in the cohomology $H^\ast(Z_j,k)_{\mathrm{red}}$, via the composite
\[
(\g_a^{(r)})^n\overset{\sigma}\to H^\ast(Z^{\ot n},k)_{\mathrm{red}}\to H^\ast(Z_j,k)_{\mathrm{red}}\cong k[\beta_j].
\]
For any such $\underline{c}\in (\g_a^{(r)})^n$ let $Def_{\underline{c}}$ denote the corresponding deformation $\Z^{\ot n}\ot_{(Z')^{\ot n}}k[\e]$, where we change base along the corresponding map $(Z')^{\ot n}\to k[\e]$, $t_i\mapsto c_i\e$.
\par

Consider such a nonzero $\underline{c}$ and take $j$ such that the $j$-th entry $c_j$ is nonzero.  We claim that the image of the corresponding class $\sigma_{\Z^{\ot n}}(\underline{c})\in H^2(Z^{\ot n},k)$ in $H^2(Z_j,k)$ is exactly the class $c_j\beta_j\in H^2(Z_j,k)$.  One way to see this is to note that the Hochschild 2-cocycle corresponding to $Def_{\underline{c}}$ is a function $F_{\underline{c}}:Z^{\ot n}\ot Z^{\ot n}\to Z^{\ot n}$ with restriction
\[
F_{\underline{c}}:Z_j\ot Z_j\to Z^{\ot n},\ \ t_j^l\ot t_j^m\mapsto c_jt_j^{(l+m)-p^r},
\]
where a negative power is considered to be $0$.  (This is just as in Lemma~\ref{lem:349}.)  Composing with the counit produces the function $\bar{F}_{c_j}:t^l\ot t^m\mapsto c_j\delta_{l+m,p^r}$.  The function $\bar{F}_{c_j}$ is equal to $c_j\bar{F}_{\alpha_j}$, where $\bar{F}_{\alpha_j}$ is as in the proof of Lemma~\ref{lem:349}, and we can consult the proof of Lemma~\ref{lem:349} again to see that $[\bar{F}_{c_j}]=c_j[\bar{F}_{\alpha_j}]=c_j\beta_j$.
\par

Upon choosing coordinates of $Z^{\ot n}$ to obtain the identification 
\[
H^\ast(Z^{\ot n},k)_{\mathrm{red}}\cong \left(\otimes_{i=1}^{n} H^\ast(Z_i,k)\right)_{\mathrm{red}}\cong k[\beta_1,\dots, \beta_n],
\]
we easily see that the reduced algebra has dimension $n$ in degree $2$.  So injectivity of $\sigma_{\mathrm{red}}$ implies that $\sigma_{\mathrm{red}}$ is an isomorphism.  Consequently, the algebra map $\sigma'_{\mathrm{red}}$ (the multiplicative extension of
$\sigma_{\mathrm{red}}$)  is an isomorphism.  As a consequence, $\sigma'_{\mathscr{Z}^{\ot n}}$ must be injective as well.
\end{proof}

\subsection{The proof of Proposition~\ref{prop:381}}

We retain our notations $\Z$ and $Z$ from above, and take also $\O=\O(\G_{(r)})$.

\begin{proof}[Proof of Proposition~\ref{prop:381}]
The identification of $H^1(\O,k)$ with $\Hom_k(m_G/m_G^2,k)$ implies that $\g=H^1(\O,k)$.  Invariance of the image of $\g^{(r)}$ follows from Corollary~\ref{cor:g2inv}.  Whence the algebra map $\wedge^\ast(\g[1])\ot S(\g^{(r)}[2])\to H^\ast(\O,k)$ of~\eqref{eq:374} is one of $\G_{(r)}$-algebras.  It remains to show that the map is a (linear) isomorphism.
\par
  
Since $\G$ is smooth, we can choose complete local coordinates $\{x_i\}_i$ at the identity to get algebra presentations 
\[
\O_{\mathrm{nat}}=\O(\G_{(r+1)})=k[x_1,\dots,x_n]/(x_i^{p^{r+1}})\text{ and }\O=k[x_1,\dots,x_n]/(x_i^{p^r}).
\]
Whence we have an algebra isomorphism $Z^{\ot n}\overset{\cong}\to \O$, $t_i\mapsto x_i$, under which the deformations $\Z^{\ot n}$ and $\O_{\mathrm{nat}}$ can be identified.  Thus the maps $\sigma_{\Z^{\ot n}}$ and $\sigma_\O$ are also identified, and we see that $\sigma'_\O:S(\g^{(r)}[2])\to H^\ast(\O,k)$ is an isomorphism modulo nilpotents by Proposition~\ref{prop:439}.
\par

Since we know abstractly that
\[
H^\ast(\O,k)=\wedge^\ast(H^1(\O,k))\ot S(V)=\wedge^\ast(\g)\ot S(V),
\]
for any vector space complement $V$ to $\wedge^2\g$ in $H^2(\O,k)$, it suffices to show that $\sigma_\O(\g^{(r)})$ is a complement to the second wedge power of $\g$.  However, this follows from the facts that $\sigma_\O'$ is an isomorphism modulo nilpotents and that the kernel of the reduction $H^2(\O,k)\to H^2(\O,k)_\mathrm{red}$ is exactly $\wedge^2H^1(\O,k)=\wedge^2\g$.  
\end{proof}

\subsection{In characteristic $2$}
\label{sect:2}

Suppose $\mathrm{char}(k)=2$ and let $\G$ be a smooth algebraic group over $k$.  Consider the $r$-th Frobenius kernel $\G_{(r)}$ with $r>1$.  In this case we have an algebra identification 
\[
\O(\G_{(r)})=k[x_1,\dots, x_n]/(x_1^{2^r},\dots,x_n^{2^r})=\otimes_{i=1}^n k[x_i]/(x_i^{2^r}).
\]
Furthermore, since $H^\ast\left(k[x]/(x^{2^r}),k\right)=k[a,b]/(a^2)$, where $\deg(a)=1$ and $\deg(b)=2$, we see that all elements in $H^1(\O(\G_{(r)}),k)$ are square zero.  Hence we can construct an algebra map
\begin{equation}\label{eq:594}
\wedge^\ast(\g[1])\ot S(\g^{(r)}[2])\to H^\ast(\O(\G_{(r)}),k)
\end{equation}
via the identification $\g=H^1(\O(\G_{(r)}),k)$ and the deformation map $\sigma_\O$, just as before.  The above proof of Proposition~\ref{prop:381} can now be repeated verbatim to arrive at

\begin{proposition}\label{prop:586}
When $\mathrm{char}(k)=2$ and $r>1$, the algebra map~\eqref{eq:594} is an isomorphism of $\G_{(r)}$-algebras.
\end{proposition}

Under these same hypotheses all proofs in Sections~\ref{sect:fg}--\ref{sect:supp} also apply verbatim.  Hence we are able to deal with these cases without any deviation in our presentation.
\par

\begin{remark}
When $\mathrm{char}(k)=2$ and $r=1$, the algebra map $S(\g[1])\to H^\ast(\O,k)$ induced by the identification $\g=H^1(\O,k)$ is an isomorphism.  The methods employed in the proof of Proposition~\ref{prop:381} show that, in this case,
\[
\sigma'_\O:S(\g^{(1)}[2])\to H^\ast(\O,k)
\]
is the Frobenius.
\par

Now, in degree $2$ we have an exact sequence of $\G_{(1)}$-representations $0\to \g^{(1)}\to S^2(\g)=H^2(\O,k)\to M\to 0$, where $M=\mrm{coker}(\sigma_\O)$.  The possible failure of this sequence to split over $\G_{(1)}$ obstructs our proof of Theorem~\ref{thm:fingen} below.  In particular, it is not apparent how one can construct the complement $\Gamma$ to $S(\g^{(1)}[2])$ employed in the proof of the aforementioned theorem.
\end{remark}
 
\section{Finite generation of cohomology}
\label{sect:fg}

We consider a smooth group scheme $\G$ and an integer $r > 0$.  As always, $\G$ is assumed to be affine of finite type over $k$.  In Theorems~\ref{thm:fingen} and~\ref{thm:fingen2} below, we prove finite generation of cohomology for the Drinfeld double $\D\G_{(r)} \equiv \O(\G_{(r)}) \# k\G_{(r)}$ of the $r$-th Frobenius kernel $\G_{(r)}$.  Our technique is to use the Grothendieck spectral sequence~\cite{grothendieck57} as in~\cite{friedlandersuslin97}.

\subsection{A spectral sequence for the cohomology of the double}
We begin with a general result.

\begin{proposition}
\label{prop:Groth}
Let $F: \cA \ \to \ \cB, \ G: \cB \ \to \ \cC$ be additive, left exact functors between abelian categories with enough injectives
and suppose that $F$ sends injective objects of $\cA$ to injective objects of $\cB$.
Assume further that $\cA$, $\cB$, $\cC$ have tensor products and that $F$, $G$ are equipped with natural maps
$F(V)\otimes F(V^\prime) \to F(V \otimes V^\prime), \quad G(W)\otimes G(W^\prime) \to G(W \otimes W^\prime)$.
Then for any pairing $V \otimes V^\prime \to V^{\prime\prime}$ there exists a pairing  of Grothendieck spectral sequence
$$\{ R^sG(R^t(F(V))) \ \Rightarrow R^{s+t}(G\circ F)(V) \} \ \otimes \ \{ R^{s^\prime}G(R^{t^\prime}(F(V^\prime))) 
\ \Rightarrow R^{s^\prime+t^\prime}(G\circ F)(V^\prime) \} $$
$$ \quad \to \ \{ R^{s^{\prime\prime}}G(R^{t^{\prime\prime}}(F(V^{\prime\prime}))) \ \Rightarrow
R^{s^{\prime\prime}+t^{\prime\prime}}(G\circ F)(V^{\prime\prime}) \}.$$
\end{proposition}

\begin{proof}
The Grothendieck spectral sequence for the composition of left exact functors between abelian categories with enough injectives,
$G \circ F: \cA \to \cB \to \cC$, arises from a Massey exact couple.   Namely, one
takes an injective resolution $V \to I^\b$ of an object of $V$ of $\cA$, and then takes a Cartan-Eilenberg resolution 
$F(I^\b) \to J^{\b,\b}$ of the cochain complex $F(I^\b)$; $J^{\b,\b}$ is a double complex of injective objects of $\cB$
which not only gives an injective resolution of each $F(I^n)$ but also of each $H^n(F(I^\b))$.  Then the Massey exact
couple is given by ``triples" $(i:D \to D, j:D \to E, k: E \to D)$,
$$ \cdots \stackrel{k}{\to} D = \bigoplus_p H^{p+q}(F^{p+1}(Tot(G(J^{*,*}))) )\stackrel{i}{\to}  \bigoplus_p D = H^{p+q}(F^p(Tot(G(J^{*,*})))) $$
$$\stackrel{j}{\to} E = \bigoplus_{p,q} H^{p+q}(F^p(Tot(G(J^{*,*})))/F^{p+1}(Tot(G(J^{*,*})))) \stackrel{k}{\to} \cdots ,$$
where $F^p((Tot(G(J^{*,*})))  = Tot(G(\oplus_{i \geq p} (J^{i,*})))$.

Assuming that $\mc{A}, \mc{B}$ and $\mc{C}$ have tensor products, a paring of objects in $\mc{A}$ gives 
rise to a pairing of Massey exact couples.   Namely, given injective resolutions $V \to I^\b, \ V^\prime \to I^{\prime \b},
\ V^{\prime\prime} \to I^{\prime\prime \b}$ and a pairing $V \otimes V^\prime \to V^{\prime\prime}$, then the usual
extension argument for the injective complex $I^{\prime\prime \b}$ tells us that there is a map of cochain complexes
$Tot(I^\b \otimes I^{\prime\b}) \to I^{\prime\prime \b}$, unique up to chain homotopy, extending this pairing.  
This, in turn, determines a pairing of bicomplexes
$G(J^{*,*}) \otimes G(J^{\prime *,*}) \to G(J^{\prime\prime *,*})$
and thus of filtered total complexes.  The pairing on exact couples takes
the expected form using the natural map
$\oplus_{s+t=n}(H^s(C^*) \otimes H^t(C^{\prime *})) \to H^{s+t}(C^*\otimes C^{\prime *})$

In~\cite{massey54}, Massey gives sufficient conditions for a pairing of exact couples to determine a pairing of 
spectral sequences (see also~\cite{friedlandersuslin97}).  The essential condition is Massey's
condition $\mu_n$ for each $n\geq0$: for $z\otimes z^\prime$ bihomogeneous in $E\otimes E^\prime$
and any $x \otimes x^\prime$ bihomogeneous in $D\otimes D^\prime$ such that $k(z) = i^n(x), \ k(z^\prime) = (i^\prime)^n(x^\prime)$,
there exists $x^{\prime\prime} \in D^{\prime\prime}$ with $k^{\prime\prime}(z\cdot z^\prime) = (i^{\prime\prime})^n(x^{\prime\prime})$ and
$j^{\prime\prime}(x^{\prime\prime}) = j(x)\cdot z^\prime + (-1)^{deg(z)}z\cdot j^\prime(x^\prime)$.
In our context, $z\otimes z^\prime \in H^{p+q}(F^p/F^{p+1}) \otimes H^{p^\prime+q^\prime}(F^{p^\prime}/F^{p^\prime+1})$
and $x\otimes x^\prime \in H^{p+q+1}(F^{p+n+1}) \otimes H^{p^\prime+q^\prime+1}(F^{p^\prime+n+1})$.  To satisfy
condition $\mu_n$, we take $x^{\prime\prime} \in H^{p+q+p^\prime + q^\prime}(F^{p+p^\prime +n+1})$ to 
be the image of $x\otimes x^\prime$ given by the pairing map.
\end{proof}

Recall that $\D\G_{(r)}/\O(\G_{(r)})$ is $k\G_{(r)}$-Galois as in~\cite{montgomery93}.  One can view this property as the condition
that $\D\G_{(r)}$ is a $k\G_{(r)}$ torsor (in the context of Hopf algebras) over $\O(\G_{(r)})$: there is a natural bijection 
$\D\G_{(r)} \otimes_{\O(\G_{(r)})} \D\G_{(r)} \to \D\G_{(r)} \otimes k\G_{(r)}$.   By normality of $\O(\mbb{G}_{(r)})$ in $\D\G_{(r)}$, for any $\D\G_{(r)}$-module $V$ on which $\O(\mbb{G}_{(r)})$ acts trivially we have $V^{\D\G_{(r)}}=V^{\G_{(r)}}$.  Hence the invariants functor for $\D\G_{(r)}$ factors
\[
\Hom_{\D\G_{(r)}}(k,-) \ = \ \Hom_{k\G_{(r)}}(k,-) \circ \Hom_{\O(\G_{(r)})}(k,-): \rep(\D\G_{(r)}) \ \to \ Vect.
\]
\begin{proposition}
\label{prop:specseq}
The above composition of functors leads to a Grothendieck spectral sequence
of $k$-algebras
\begin{equation}
\label{specseq-k}
E_2^{s,t}(k) = H^s(\G_{(r)},H^t(\O(\G_{(r)}),k)) \ \Rightarrow \ H^{s+t}(\D\G_{(r)},k). 
\end{equation}
For any $\D\G_{(r)}$-module $M$, the above composition of functors leads to Grothendieck spectral sequence
\begin{equation}
\label{specseq-M}
E_2^{s,t}(M) = H^s(\G_{(r)},H^t(\O(\G_{(r)}),M)) \ \Rightarrow \ H^{s+t}(\D\G_{(r)},M) 
\end{equation}
which is a spectral sequence of modules over (\ref{specseq-k}).
\end{proposition}

\begin{proof}
The equalities
$$\Hom_{\O(\G_{(r)})}(k,(\D\G_{(r)})^*) \ = \ \Hom_{\O(\G_{(r)})}(\D\G_{(r)},k) \ = \ \Hom_k(k\G_{(r)},k) = (k\G_{(r)})^*$$
imply $\Hom_{\O(\G_{(r)})}(k,(\D\G_{(r)})^*)$ is  projective as well as  injective as a $k\G_{(r)}$-module (because a $k\G_{(r)}$ module is projective 
if and only if it is injective~\cite{jantzen07,montgomery93}).   Since $\Hom_{\O(\G_{(r)})}(k,(\D\G_{(r)})^*) = (k\G_{(r)})^*$, we conclude that $\Hom_{\O(\G_{(r)})}(k,-)$
sends injective $\D\G_{(r)}$-modules to injective $k\G_{(r)}$-modules.  Consequently, Grothendieck's construction of the spectral sequence
for a composition of left exact functors applies to the 
composition $\Hom_{k\G_{(r)}}(k,-) \circ \Hom_{\O(\G_{(r)})}(k,-)$, and this spectral sequence takes the form~\eqref{specseq-k} when applied
to $k$ and the form~\eqref{specseq-M} when applied to $M$.

The algebra structure on~\eqref{specseq-k} and the module structure on~\eqref{specseq-M} follow from the multiplicative structure
established in Proposition \ref{prop:Groth} in view of the pairing $k \otimes k \to k$ (mulitplication of $k$) and $k \otimes M \to M$
(pairing with the trivial module) in the category $\rep(\D\G_{(r)})$.
\end{proof}

\subsection{Finite generation}

We can now prove that $H^*(\D\G_{(r)},k)$ is a finitely generated algebra.  This will be followed by Theorem \ref{thm:fingen2}, establishing our general finite generation theorem.  Recall that an algebra map $A\to B$ is called finite if $B$ is a finite module over $A$.  Recall also the map $\theta_r$ of Definition~\ref{def:g2HD}.

\begin{theorem}
\label{thm:fingen}
Let $\G$ be a smooth group scheme over a field $k$ of characteristic $p > 0$, let $r > 0$ be a positive integer, and let $\D\G_{(r)} \equiv \O(\G_{(r)}) \# k\G_{(r)}$ denote the Drinfeld double of the $r$-th Frobenius kernel of $\G$.

Then the graded $k$-algebra map
$$\theta_r:H^*(\G_{(r)},k) \otimes S(\g^{(r)}[2]) \ \to \ H^*(\D\G_{(r)},k)$$
is finite.
 
Consequently, 
\begin{itemize}
\item
$H^*(\D\G_{(r)},k)$ is a finitely generated $k$-algebra.  
\item
For any finite dimensional $\D\G_{(r)}$-module $M$ whose restriction to $\O(\G_{(r)})$ has trivial action, $H^*(\D\G_{(r)},M)$ is a finite 
$H^*(\G_{(r)},k) \otimes S(\g^{(r)}[2])$-module and hence a finite $H^*(\D\G_{(r)},k)$-module.
\end{itemize}
\end{theorem}

\begin{proof}
Take $\O=\O(\G_{(r)})$ and $C^*=H^*(\G_{(r)},k) \otimes S(\g^{(r)}[2])$.  This proof is an adaption of the proof of Theorem 1.1 of~\cite{friedlandersuslin97}.  Let $\{ E^{s,t}_r, r \geq 2 \}$ denote the spectral seqence
$\{ E^{s,t}_r(k), r \geq 2 \}$ of Proposition \ref{prop:specseq}.   

Observe that $H^*(\O,k) = S(\g^{(r)}[2]) \otimes \Gamma$; here, $S(\g^{(r)}[2])$ has trivial $\G_{(r)}$-action
and $\Gamma \equiv \wedge^*(H^1(\O,k))$ is finite dimensional.  Thus, $E^{*,*}_2 = H^*(\G_{(r)},H^*(\O,k))$ equals
$H^*(\G_{(r)},\Gamma) \otimes S(\g^{(r)}[2])$, since $M \mapsto H^0(\G_{(r)},M \otimes V)$ is the composite of $H^0(\G_{(r)},-)$
and the exact functor $- \otimes V$ for any trivial $\G_{(r)}$-module $V$.
We equip $H^*(\G_{(r)},H^*(\O,k)) = H^*(\G_{(r)},\Gamma) \otimes S(\g^{(r)}[2])$ with 
the ``external tensor product module structure" for the algebra 
$C^* \ = \      H^*(\G_{(r)},k)  \otimes S(\g^{(r)}[2])$.

By Theorem 1.1 of~\cite{friedlandersuslin97}, $H^*(\G_{(r)},\Gamma)$ is a finite $H^*(\G_{(r)},k)$-module.  It follows that $H^*(\G_{(r)},H^*(\O,k))$ is a finite $C^*$-module.
We identify this $C^*$-module structure on  $H^*(\G_{(r)},H^*(\O,k))$ as that given by the coproduct
$\phi \otimes \psi$ of two maps $\phi, \ \psi$ associated to the spectral sequence:  The first is the map 
$$\phi: S(\g^{(r)}[2]) \ \to \ H^*(\D\G_{(r)},k) = E_\infty^* \ \to \ E^{0,*}_\infty \subset E^{0,*}_2 \subset E^{*,*}_2$$
given by Proposition \ref{prop:g2HD};  the second is the natural map 
$$ H^{ev}(\G_{(r)},k) \subset E_2^{*,0} \ \subset \ E^{*,*}_2.$$

We have thus verified the hypotheses of Lemma 1.6 of~\cite{friedlandersuslin97}, enabling us to conclude that $H^*(\D\G_{(r)},k)$
is a finite module over the finitely generated algebra $C^*$ and thus is itself finitely generated.

Now, we consider a finite  dimensional $\D\G_{(r)}$-module $M$ whose restriction to $\O$ has trivial action.  
Then $E^{*,*}_2(M) = H^*(\G_{(r)},H^*(\O,M))$ equals $H^*(\G_{(r)},\Gamma \otimes M) \otimes S(\g^{(r)}[2])$
which is a finite $C^*$-module by another application of Theorem 1.1 of~\cite{friedlandersuslin97}
(this time, for the finite dimensional $k\G_{(r)}$-module $\Gamma \otimes M$).
Since $\{ E^{*,*}_r(M) \}$ is a module over $\{ E^{*,*}_r\}$ by Proposition \ref{prop:Groth}, Lemma 1.6 of~\cite{friedlandersuslin97} applies once again to imply that  $H^*(\D\G_{(r)},M)$ is finite as a $C^*$-module and thus as a $H^*(\D\G_{(r)},k)$-module.
\end{proof}

Recall our notation $|A| \ \equiv \Spec\!\ H^{ev}(A,k)_{\mathrm{red}}$ from the introduction.

\begin{corollary}
We have the inequality
\[
\dim|\D\G_{(r)}| \ \leq \ \dim |k\G_{(r)}|+\dim\mathbb{G}.
\]
\end{corollary}

\begin{proof}
Since $\theta_r$ is finite, by Theorem~\ref{thm:fingen}, the induced map on affine spectra 
\[
|\D\G_{(r)}| \to \Spec\left(H^{ev}(\G_{(r)},k)_\mathrm{red}\otimes S(\g^{(r)}[2])\right) \ \cong \ |k\mathbb{G}_r|\times \mathbb A^d
\]
has finite fibers, where $d = \dim(\G)$.  
\end{proof}

\begin{proposition}
\label{prop:simpmod}
Let $G$ be an infinitesimal group scheme.  If $V$ is a simple module for $\D G$, then $V$ restricts to a trivial $\O(G)$-module.
\end{proposition}

\begin{proof}
The maximal ideal $m$ in $\O(G)$ is nilpotent and is preserved by the adjoint action of $G$ on $\O(G)$.  Hence, the ideal $kG\cdot m \subset \D G$ is also nilpotent, and therefore contained in the Jacobson radical of $\D G$.  We conclude that restricting along the projection $\D G \to \D G/(kG\cdot m)  = kG$ determines a bijection $\mathrm{Irrep}(kG)\to \mathrm{Irrep}(\D G)$.
\end{proof}

In the following theorem, we implicitly use the following fact for any Noetherian $k$-algebra $C$: a $C$-module $M$ is Noetherian if and only if it is finitely generated (as a $C$-module).

\begin{theorem}
\label{thm:fingen2}
As in Theorem~\ref{thm:fingen}, let $\G$ be a smooth group scheme over a field $k$ 
of characteristic $p > 0$, let $r > 0$ be a positive integer, and let $\D\G_{(r)} \equiv \O(\G_{(r)}) \# k\G_{(r)}$ denote the Drinfeld double of the $r$-th Frobenius kernel of $\G$.

If $M$ is a finite dimensional $\D\G_{(r)}$-module, then $H^*(\D\G_{(r)},M)$ is finitely generated as a $H^*(\D\G_{(r)},k)$-module.
\end{theorem}

\begin{proof}
By Theorem \ref{thm:fingen} and Proposition \ref{prop:simpmod}, $H^*(\D\G_{(r)},M)$ is finitely generated over $H^*(\D\G_{(r)},k)$
whenever $M$ is an irreducible $\D\G_{(r)}$ module.  More generally, we proceed by induction on the length of a composition
series for $M$ as a $\D\G_{(r)}$-module.   Consider a short exact sequence $0 \to N \to M \to Q \to 1$ of finite dimensional
$\D\G_{(r)}$-modules with $N$ irreducible and assume our induction hypothesis applies to $Q$.   
Let $V \subset H^*(\D\G_{(r)},M)$ denote the image of $H^*(\D\G_{(r)},N)$ and
let $W \subset H^*(\D\G_{(r)},Q)$ denote the image of $H^*(\D\G_{(r)},M)$.  Since $H^*(\D\G_{(r)},k)$ is Noetherian, $V$ is a Noetherian
$H^*(\D\G_{(r)},k)$-module since it is a quotient of the Noetherian $H^*(\D\G_{(r)},k)$-module  $H^*(\D\G_{(r)},N)$; 
moreover, $W$ is a Noetherian
$H^*(\D\G_{(r)},k)$-module since it is a submodule of the Noetherian $H^*(\D\G_{(r)},k)$-module  $H^*(\D\G_{(r)},Q)$.  Granted the short exact 
sequence \ $0 \to V \to H^*(\D\G_{(r)},M) \to W \to 0$ of $H^*(\D\G_{(r)},k)$-modules, we conclude that 
$H^*(\D\G_{(r)},M)$ is also a Noetherian as a $H^*(\D\G_{(r)},k)$-module.
\end{proof}

\subsection{Cohomology of relative doubles}\label{sect:relative}
Given an inclusion of finite dimensional Hopf algebras $A\to B$, we can form the relative double $\D(B,A)$, which is the vector space $B^\ast\ot A$ along with multiplication given by the same formula as for the standard double.  Rather, we give $\D(B,A)$ the unique Hopf structure so that the vector space inclusion $\D(B,A)\to \D(B)$ is a map of Hopf algebras.  The relative double can be of technical importance, especially in tensor categorical settings (see for example~\cite{gelakinaidunikshych09,etingofnikshychostrik11}).
\par

For a closed subgroup $G\to \mbb{G}_{(r)}$ we write $\D(\mbb{G}_{(r)},G)$ for the relative double
\[
\mrm{D}(\mbb{G}_{(r)},G)=\mrm{D}(k\mbb{G}_{(r)},kG)=\O(\mbb{G}_{(r)})\# kG,
\]
where the smash product is taken relative to the adjoint action of $G$ on $\O(\mbb{G}_{(r)})$.
\par

Dually, for a quotient $B\to C$ of finite dimensional Hopf algebras we define the relative double $\D(C,B)$ as the vector space $C^\ast\ot B$ along with the unique Hopf structure so that the inclusion $\D(C,B)\to \D(B^\ast)$ is a map of Hopf algebras.  For a group scheme quotient $\mbb{G}_{(r)}\to G'$ we write
\begin{equation}\label{eq:885}
\mrm{D}(G',\mbb{G}_{(r)})=\O(G')\# k\mbb{G}_{(r)}.
\end{equation}
From~\cite[Eq.\ (11)--(12)]{radford93}, we see that $\D(G',\G_{(r)})$ is identified with the relative double $\D(kG',k\mbb{G}_{(r)})$.

\begin{theorem}
Let $\mbb{G}$ be a smooth algebraic group.  Consider an arbitrary closed subgroup scheme $G$ in $\mbb{G}_{(r)}$, and the relative double $\D(\mbb{G}_{(r)},G)$.  Then,
\begin{itemize}
\item The cohomology $H^\ast(\D(\mbb{G}_{(r)},G),k)$ is a finitely generated algebra.
\item If $M$ is a finite dimensional $\D(\mbb{G}_{(r)},G)$-module, then $H^\ast(\D(\mbb{G}_{(r)},G),M)$ is a finitely generated module over $H^\ast(\D(\mbb{G}_{(r)},G),k)$.
\end{itemize}
The same finite generation results hold for the relative doubles $\mrm{D}(\mbb{G}_{(r)}/\mbb{G}_{(s)},\mbb{G}_{(r)})$, for $s\leq r$.
\end{theorem}

\begin{proof}[Sketch proof]
Consider a closed subgroup $G\to \mbb{G}_{(r)}$.  We have the sequence $\O(\mbb{G}_{(r)})\to D(\mbb{G}_{(r)},G)\to kG$, from which we derive Grothendieck spectral sequences as in~\eqref{specseq-k} and~\eqref{specseq-M}.  We need to exhibit a finitely generated algebra of permanent cocycles in the $E_2$-page of the spectral sequence
\[
E_2^{s,t}(k) = H^s(G,H^t(\O(\G_{(r)}),k)) \ \Rightarrow \ H^{s+t}(\D(\mbb{G}_{(r)},G),k)
\]
over which $E_2^{\ast,\ast}$ is a finite module.  Just as in the proof of Theorem~\ref{thm:fingen}, it suffices to show that the image of the embedding $\sigma_\O:\g^{(r)}\to H^2(\O(\G_{(r)}),k)$ from Section~\ref{sect:HO} consists entirely of permanent cocycles in $E_2^{\ast,\ast}$.  The deformation $\mrm{D}_\mrm{nat}=\mrm{D}(\mbb{G}_{(r+1)},G)$ provides a lifting $\sigma_\D:\g^{(r)}\to H^2(\D(\mbb{G}_{(r)},G),k)$ of $\sigma_\O$, which verifies permanence of the cocycles $\g^{(r)}\subset H^2(\O(\mbb{G}_{(r)}),k)$.  We can now argue as in the proof of Theorem~\ref{thm:fingen} to establish finite generation.
\par

In the case of a quotient $\mbb{G}_{(r)}/\mbb{G}_{(s)}\cong \mbb{G}_{(r-s)}^{(s)}$, we have the deformation $\O_\mrm{nat}=\O(\mbb{G}_{(r+1)}/\mbb{G}_{(s)})$ of $\O(\mbb{G}_{(r)}/\G_{(s)})$ and the deformation $\mrm{D}_\mrm{nat}=\O_\mrm{nat}\# k\mbb{G}_{(r)}$ of the relative double $\D(\mbb{G}_{(r)}/\mbb{G}_{(s)},\mbb{G}_{(r)})$.  These deformations provide an inclusion
\[
\sigma_\O:\g^{(r)}\to H^2(\O(\mbb{G}_{(r)}/\mbb{G}_{(s)}),k)
\]
and a lifting $\sigma_\D:\g^{(r)}\to H^2(\mrm{D}(\mbb{G}_{(r)}/\mbb{G}_{(s)},\mbb{G}_{(r)}),k)$ of $\sigma_\O$.  We employ $\sigma_\O$ and $\sigma_\D$, and again argue as in Theorem~\ref{thm:fingen}, to establish finite generation.
\end{proof}

\begin{remark}
For a general quotient $p:\mbb{G}_{(r)}\to G'$, we expect that finite generation of cohomology for the relative double $\D(G',\mbb{G}_{(r)})$ can be proved via the same deformation theoretic approach as above.  If we take $K=\ker(p)$, the necessary deformation in this case should be provided by the quotient scheme $\mbb{G}_{(r+1)}/K$.  Some care needs to be taken, however, in dealing with the arbitrary nature of the subgroup $K$.
\end{remark}

\begin{remark}
In the notation of~\cite[Sect.\ 2B]{gelakinaidunikshych09}, the relative double $\D(\mbb{G}_{(r)},G)$ has representation category isomorphic to the relative center $Z_{\mc{C}}(\mc{M})$ where $\mc{C}=\mrm{rep}(\mbb{G}_{(r)})$, $\mc{M}=\mrm{rep}(G)$, and the $\mc{C}$-action on $\mc{M}$ is given by the restriction functor $\mrm{rep}(\mbb{G}_{(r)})\to \mrm{rep}(G)$.  Similarly, for a quotient $\mbb{G}_{(r)}\to G'$, we have $\mrm{D}(G',\mbb{G}_{(r)})\cong Z_\mc{D}(\mc{N})$ where $\mc{D}=\mrm{corep}(k\mbb{G}_{(r)})$ and $\mc{N}=\mrm{corep}(k G')$.
\end{remark}

In the final two sections of this paper we provide analyses of the spectrum of cohomology and support for the (usual) double $\D\G_{(r)}$.  These analyses are valid for the relative doubles $\D(\mbb{G}_{(r)},G)$ as well.  In particular, one replaces $k\G_{(r)}$ with $kG$ and repeats the arguments verbatim.  As we would like to emphasize the double $\D\G_{(r)}$, we choose not to make explicit reference to the relative settings therein.

\section{Spectrum of cohomology for classical groups}
\label{sect:qlog}

By Theorem~\ref{thm:fingen}, the cohomology of the double $\D\G_{(r)}$ is finite over the image of $H^\b(\G_{(r)},k)\ot S(\g^{(r)}[2])$, under the map $\theta_r$ of Definition~\ref{def:g2HD}.  The map $\theta_r$ then induces a finite scheme map
\begin{equation}\label{eq:mapinquestion}
\Theta_r:|\D\G_{(r)}|\to |k\G_{(r)}|\times (\g^\ast)^{(r)},
\end{equation}
where $|A|=\Spec\!\ H^{ev}(A)_{\mathrm{red}}$.
\par

In this section we show that $\Theta_r$ is an isomorphism when $\G$ is one of many classical algebraic groups with either $p$ sufficiently
large for $p$ or $r$ sufficiently large relative to the dimension of $\G$.  Our results follow from an analysis of algebraic groups which admit a quasi-logarithm.

\begin{remark}
The schemes $|k\G_{(r)}|$ have been extensively studied and, in conjunction with support varieties of $\G_{(r)}$-representations, provide one means of approaching modular representation theory.  One can see the survey~\cite{pevtsova} for example, and the references therein.
\end{remark}

\subsection{Quasi-logarithms}
\label{sect:ql}

Let $\G$ be an algebraic group with Lie algebra $\g=\mathrm{Lie}(\G)$.  We let $\G$ act on itself and its Lie algebra $\g$ via the adjoint action.  The following definition is adapted from~\cite{kazhdanvarshavsky06}.

\begin{definition}
A quasi-logarithm for $\G$ is a $\G$-equivariant map $L:\G\to \g$ of $k$-schemes such that $L(1)=0$ and the differential $\mathrm{d}_1L:T_1\G\to T_0\g$ is the identity on $\g$.
\end{definition}

The information of a quasi-logarithm for $\G$ is exactly the information of a $\G$-linear splitting $\g^\ast\to m_\G$ of the projection $m_\G\to m_\G/m_\G^2=\g^\ast$, where $m_\G$ is the maximal ideal corresponding to the identity of $\G$.  Let us give some examples.

\begin{proposition}\label{prop:ql1}
The general linear group $\GL_n$ admits a quasi-logarithm.
\end{proposition}

\begin{proof}
The augmentation ideal $m_{\GL}$ is generated by the functions $x_{ij}-\delta_{ij}$.  Take $V$ to be the span of these functions $k\{x_{ij}-\delta_{ij}:1\leq i,j\leq n\}$.  The sequence $V\to m_{\GL}\to \gl_n^\ast$ provides a linear isomorphism between $V$ and $\gl_n^\ast$.
\par

For the comultiplication on $\O(\GL_n)$ we have $\Delta(x_{ij})=\sum_k x_{ik}\ot x_{kj}$.  Thus for the adjoint coaction $\rho$ restricted to $V$ we will have $\rho(V)\subset (k1_\O\oplus V)\otimes \O(\GL_n)$.  Since $m_{\GL}$ is preserved by the adjoint coaction, and $V\subset m_{\GL}$, we will also have $\rho(V)\subset m_{\GL}\ot \O(\GL_n)$.  Taking the intersection of these two subspaces gives $\rho(V)\subset V\ot \O(\GL_n)$.  Thus we see that $V$ is a subcomodule of $\O(\GL_n)$ under the adjoint coaction.  The aforementioned sequence then provides a $\GL_n$-linear isomorphism $V\to \gl_n^\ast$.  Taking the inverse $\gl_n^\ast\to V\subset m_{\GL}$ provides the desired quasi-logarithm.
\end{proof}

We can also address many simple algebraic groups.  An odd prime $p$ is {\it very good} for a simple algebraic group $\G$ if $p$ does not divide $n$ for $\G$ of type 
$A_{n-1}$, if $p \neq 3$ for $\G$ of type $E_6,\ E_7,\ F_4,\ G_2$, and $p \neq 3,5$ for $\G$ of type $E_8$.  For convenience we extend the notion of a very good prime to $\GL_n$, in which case all primes will be considered very good.

\begin{corollary}[{cf.~\cite[Lem.\ C3]{bkv16}}]\label{cor:ql2}
If $\G$ is a simple algebraic group for which $p$ is very good, then $\G$ admits a quasi-logarithm.  Furthermore, any Borel subgroup $\mathbb{B}$ in such a $\G$ also admits a quasi-logarithm.
\end{corollary}

\begin{proof}
In this case there exists an integer $n$ and an embeddging $i:\G \to \GL_n$ such that the differential $\mathrm{d}_1i:\g \to \gl_n$ admits a $\G$-equivariant splitting $\tau:\gl_n \to \g$, by a result of Garibaldi~\cite[Prop.\ 8.1]{garibaldi09}.  Composing with a quasi-logarithm $L$ for $\GL_n$ produces a quasi-logarithm $L'$ for $\G$,
\[
\G\to \GL_n\overset{L}\to \gl_n\overset{\tau}\to \g.
\]
By~\cite[Lem.\ 1.8.3]{kazhdanvarshavsky06}, the restriction of $L'$ to any Borel subgroup $\mathbb{B}$ will provide a quasi-logarithm for $\mathbb{B}$.
\end{proof}

Consider a semisimple algebraic group $\G$ and a unipotent subgroup $\mbb{U}$ in $\mbb{G}$ which is normalized by a maximal torus.  We let $cl(\mbb{U})$ denote the nilpotence class of a $\mbb{Q}$-form of $\mbb{U}$ in a $\mbb{Q}$-form of $\G$ (see~\cite{seitz00}).  For example, if we consider $\G=\mathrm{SL}_n$ and $\mbb{U}$ the unipotent subgroup of upper triangular matrices, then $cl(\mbb{U})=n-1$.  The following result is covered in work of Seitz.

\begin{proposition}[{\cite[Prop.\ 5.2]{seitz00}}]\label{prop:ql3}
Let $\G$ be semisimple and $\mbb{U}$ be a unipotent subgroup in $\G$ which is normalized by a maximal torus.  If $p>cl(\mbb{U})$ then $\mbb{U}$ admits a quasi-logarithm.
\end{proposition}

The main principle here is quite simple.  Under this restriction on $p$, the usual exponent on the $\mbb{Q}$-form $\exp_\mbb{Q}:\mtt{u}_\mbb{Q}\to \mbb{U}_\mbb{Q}$ is an isomorphism defined over $\mbb{Z}_{(p)}$, and hence induces an isomorphism $\exp_k:\mtt{u}\to \mbb{U}$ over $k$.  We define $L$ as the inverse $L=\exp_k^{-1}$.  Equivariance of $L$ under the adjoint $\mbb{U}$-action follows from $\mbb{U}_{\mbb{Q}}$-invariance of $\exp_\mbb{Q}$.

\subsection{Induced gradings on the double}

Consider an algebraic group $\G$ with a fixed quasi-logarithm $L$.  From $L$ we get a map of $\G_{(r)}$-algebras $S(\g^\ast)\to \O(\G_{(r)})$, for each $r$, via the composition $S(\g^\ast)\overset{L^\ast}\to \O(\G)\to \O(\G_{(r)})$.  Since each $x\in \g^\ast$ maps into the augmentation ideal in $\O(\G)$, there is furthermore an induced $\G_{(r)}$-algebra map $l_r:S(\g^\ast)/I_r\to \O(\G_{(r)})$, where $I_r$ is the ideal generated by the $p^r$-th powers of elements in $\g^\ast$.  Rather, $I_r$ is the ideal generated by the image of the augmentation ideal under the $r$-th Frobenius.  We can now take a smash product to arrive at a final algebra map
\begin{equation}\label{eq:Lr}
\mathscr{L}_r:\left(S(\g^\ast)/I_r\right)\# k\G_{(r)}\to \D\G_{(r)}.
\end{equation}
\par

We note that the algebra $S(\g^\ast)/I_r$ is graded, since the ideal $I_r$ is generated by the homogenous elements $x^{p^r}$, $x\in \g^\ast$.  Furthermore, under this grading $k\G_{(r)}$ acts by graded endomorphisms.  Hence the smash product $\left(S(\g^\ast)/I_r\right)\# k\G_{(r)}$ is graded with $\g^\ast$ in degree $1$ and $k\G_{(r)}$ in degree $0$.  This point will be of some significance below.

\begin{lemma}\label{lem:grading}
Suppose $\G$ is smooth and admits a quasi-logarithm $L$.  Then for any $r>0$ the above map 
$\mathscr{L}_r:\left(S(\g^\ast)/I_r\right)\# k\G_{(r)}\to \D\G_{(r)}$ is an isomorphism of algebras.
\end{lemma}

\begin{proof}
Recall that $\dim(\G) = \dim(\g)$ whenever $\G$ is smooth (see \cite[I.7.17(1)]{jantzen07}).
The localization at the distinguished maximal ideals of $S(\g^\ast)$ and $\O(\G)$, $S(\g^\ast)_0 \to \O(\G)_1$
is a local map of regular, local $k$-algebras of dimension $\dim\g$ which induces an isomorphism on corresponding
maximal ideals modulo their squares.  Thus, $L$ induces an isomorphism of complete local rings
$\hat L_1:\widehat{S(\g^\ast)}\overset{\cong}\to \widehat{\O_{\G,1}}$ (see e.g.~\cite[proof of Lem.\ 10.28.1]{matsumura}). 
We mod out by the images of the maximal ideals under the $r$-th Frobenius to arrive at an isomorphism
\[
S(\g^\ast)/I_r=\widehat{S(\g^\ast)}/\hat{I_r}\overset{\cong}\longrightarrow 
\widehat{\O_{\G,1}}/(f^{p^r}:f\in \hat{m}_{\G})=\O(\G)/(f^{p^r}:f\in m_\G)=\O(\G_{(r)}).
\]
One can check on elements to see that the above isomorphism is exactly $l_r$.
Thus, $l_r:S(\g^\ast)/I_r\to \O(\G_{(r)})$ and hence $\mathscr{L}_r:\left(S(\g^\ast)/I_r\right)\# k\G_{(r)}\to \D\G_{(r)}$
are isomorphisms.
\end{proof}

As a consequence of Lemma~\ref{lem:grading}, we see that when $\G$ is smooth and admits a quasi-logarithm the double $\D\G_{(r)}$ inherits a grading induced by $\mathscr{L}_r$.  This grading is such that $k\G_{(r)}$ lies in degree $0$ and $\mathscr{L}_r(\g^\ast)$ lies degree $1$.  The coordinate algebra $\O(\G_{(r)})$ will be a graded subalgebra in the double, with $\O(\G_{(r)})_0=k$ and $\O(\G_{(r)})_1=l_r(\g^\ast)$.
\par

We now consider the algebras $\O(\G_{(r)})$ and $\D\G_{(r)}$ as graded (Noetherian, locally finite) algebras.  As with any Noetherian graded algebra, the cohomologies $\Ext^\b_{\O(\G_{(r)})}(M,N)$ and $\Ext^\b_{\D\G_{(r)}}(M,N)$ of finitely generated graded modules inherit natural gradings, in addition to the cohomological gradings.  In particular, the cohomologies $H^\b(\O(\G_{(r)}),k)$ and $H^\b(\D\G_{(r)},k)$ will be graded.  (See e.g.~\cite{atv07}.)  We call this extra grading on cohomology the {\it internal grading}.

\begin{lemma}\label{lem:HOdeg}
Let $\G$ be smooth with a fixed quasi-logarithm.  Consider $H^\b(\O(\G_{(r)}),k)$ with its induced internal grading.  Under the isomorphism $\wedge^\b(\g)\ot S(\g^{(r)}[2])\cong H^\b(\O(\G_{(r)}),k)$ of Proposition~\ref{prop:381}, $\g$ is identified with a subspace of internal degree $1$ and $\g^{(r)}$ is identified with a subspace of internal degree $p^r$. 
\end{lemma}

\begin{proof}
The algebra $\O=\O(\G_{(r)})$ is connected graded and generated in degree $1$.  Hence $\g\cong H^1(\O,k)$ is concentrated in degree $1$ (see~\cite{atv07}).
\par

Under the gradings induced by the quasi-logarithm, the reduction
\[
\O_{\mathrm{nat}}=\O(\G_{(r+1)})\to \O
\]
is a homogeneous map, and each deformation $Def_\xi=\O_{\mathrm{nat}}\ot_{\O(\G_{(r+1)}/\G_{(r)})}k[\e]$ associated to an element $\xi\in \g^{(r)}$ is graded, where we take $\deg(\e)=p^r$.  By choosing any graded $k[\e]$-linear identification $\O[\e]\cong Def_\xi$ we see that the associated function $F_\xi:\O\ot \O\to \O$, which is defined by the equation $a\cdot_\xi b=ab+F(a,b)\varepsilon$, is such that $\deg(F(a,b))=\deg(a\ot b)-p^{r}$.  So the Hochschild $2$-cocycle $F_\xi\in \Hom_k(\O\ot \O,\O)$ is degree $p^r$, as is
 its image $\bar{F}_\xi\in \Hom_k(\O\ot\O,k)$.  It follows that $\sigma_\O(\xi)=[\bar{F}_\xi]\in H^2(\O,k)$ is a homogeneous element of degree $p^r$.
\end{proof}

\subsection{Spectra of cohomology}

Recall the map $\Theta_r$ from~\eqref{eq:mapinquestion}, and the definition $|A|=\Spec\!\ H^{ev}(A,k)_{\mrm{red}}$.

\begin{theorem}\label{thm:|log|}
Suppose $\G$ is a smooth algebraic group which admits a quasi-logarithm.  If $r$ is such that $p^r>\dim(\G)$, then 
\[
\theta_r: H^\ast(\G_{(r)},k)\ot S(\g^{(r)}[2])\to H^\ast(\D\G_{(r)},k)
\]
is finite and injective.  Consequently, the scheme map
\[
\Theta_r:|\D\G_{(r)}|\to |k\G_{(r)}|\times (\g^\ast)^{(r)}
\]
is finite and surjective, and furthermore $\dim |\D\G_{(r)}|=\dim|k\G_{(r)}|+\dim\G$.
\end{theorem}

\begin{proof}
We freely use the notation of the proof of Theorem~\ref{thm:fingen}, and omit the shift $[2]$ in the symmetric algebra to ease notation.  According to Lemma~\ref{lem:grading}, $\D\G_{(r)}$ inherits a natural algebra grading via the isomorphism $\mathscr{L}_r$ of~\eqref{eq:Lr}, $\O$ is a graded subalgebra, and the exact sequence $1\to \O\to \D\G_{(r)}\to k\G_{(r)}\to 1$ is a sequence of graded algebra maps, where $k\G_{(r)}$ is taken to be entirely in degree $0$.  In this case the spectral sequence of Proposition~\ref{prop:specseq} inherits an internal grading so that all differentials are homogeneous of degree $0$.
\par

The internal grading at the $E_2$-page is such that the degree on each $E_2^{ij}=H^i(\G_{(r)},H^j(\O))$ is induced by the degree on $H^j(\O)$.  In particular, each summand $\wedge^{j_1}\g\ot S^{j_2}(\g)\subset H^{j_1+j_2}(\O)$ is of internal degree $j_1+p^rj_2$, by Lemma~\ref{lem:HOdeg}, and the corresponding summands in the decomposition
\[
\begin{array}{rl}
H^i(\G_{(r)},H^j(\O,k))&=H^i\left(\G_{(r)},\bigoplus_{j_1+2j_2=j}\wedge^{j_1}(\g)\ot S^{j_2}(\g^{(r)})\right)\\
&=\bigoplus_{j_1+2j_2=j}H^i\left(\G_{(r)},\wedge^{j_1}(\g)\right)\ot S^{j_2}(\g^{(r)})
\end{array}
\]
are of respective degrees $j_1+j_2p^r$.
\par

Since $\dim\g<p^r$, the index $j_1$ is such that $0\leq j_1<p^r$.  Hence the degree $p^r\mathbb{Z}$ portion of the $E_2$-page is exactly the prescribed subalgebra of permanent cocycles
\begin{equation}\label{eq:888}
(E^{i,j}_2)_{p^r\mathbb{Z}}=H^i(\G_{(r)},k)\ot S^{j/2}(\g^{(r)})\Rightarrow \left(H^\ast(\D\G_{(r)},k)\right)_{p^r\mathbb{Z}},
\end{equation}
where $S^{j/2}(\g^{(r)})$ is taken to be $0$ when $j$ is odd.
\par

By homogeneity of the differentials, and the fact that all of the elements of degrees $p^r\mathbb{Z}$ in $E_2^{\ast,\ast}$ are cocycles by~\eqref{eq:888}, we see that no elements of degrees $p^r\mathbb{Z}$ are coboundaries.  One can make the same argument at each subsequent page of the spectral sequence to find that that the map $H^i(\G_{(r)},k)\ot S^t(\g^{(r)})\to E_s^{i,2t}$ is injective for all $i,\ t$, and $s$.  It follows that $\gr\theta_r:H^\b(\G_{(r)},k)\ot S(\g^{(r)})\to E^{\ast,\ast}_\infty$ is injective.
\par

Injectivity of the associated graded map $\gr\theta_r$ implies that $\theta_r:H^\b(\G_{(r)},k)\ot S(\g^{(r)})\to H^\b(\D\G_{(r)},k)$ is injective.  By Theorem~\ref{thm:fingen}, $\theta_r$ is also finite.  After taking even degrees and reducing,
\[
\theta^{ev}_{\mathrm{red}}:H^{ev}(\G_{(r)},k)_{\mathrm{red}}\ot S(\g^{(r)})\to H^\b(\D\G_{(r)},k)_{\mathrm{red}}
\]
remains injective and finite.  In particular, $\theta^{ev}_{\mathrm{red}}$ is an integral extension.  Thus, the map on spectra induced
by $\theta^{ev}$ is finite and surjective~\cite[Thm.\ 9.3]{matsumura}.  The asserted computation of dimension follows.
\end{proof}

Note that the dimension of $|\O(\G_{(r)})|$ is equal to $\dim\G$, by Proposition~\ref{prop:381} (and~\cite[I.7.17(1)]{jantzen07}).  Hence the equality of dimensions of Theorem~\ref{thm:|log|} can also be written as
\[
\dim |\D\G_{(r)}|=\dim|k\G_{(r)}|+\dim|\O(\G_{(r)})|.
\]
Under stronger assumptions on $p$ or $r$ we can significantly strengthen the conclusion of Theorem~\ref{thm:|log|}.  Indeed one can leverage the internal grading on the given spectral sequence, as in the proof of Theorem~\ref{thm:|log|}, to show that $\Theta_r$ is an isomorphism in such circumstances.

\begin{theorem}\label{thm:|log|2}
Suppose $\G$ is a smooth algebraic group which admits a quasi-logarithm.  Suppose additionally that $r$ is such that $p^r>2\dim\G$.  Then the image of the injective algebra map
\[
\theta_r: H^\ast(\G_{(r)},k)\ot S(\g^{(r)}[2])\to H^\ast(\D\G_{(r)},k)
\]
admits an $H^\ast(\G_{(r)},k)\ot S(\g^{(r)}[2])$-module complement $J$ which consists entirely of nilpotent elements in $H^\ast(\D\G_{(r)},k)$.  Furthermore, the induced map on reduced spectrums
\[
\Theta_r:|\D\G_{(r)}|\to |k\G_{(r)}|\times (\g^\ast)^{(r)}
\]
is an isomorphism.
\end{theorem}

\begin{proof}
Fix a quasi-logarithm on $\G$, and consider the induced gradings on cohomology.  Note that we can consider all of our $\mathbb{Z}$-graded spaces as $\mbb{Z}/p^r\mbb{Z}$-graded spaces, via the projection $\mbb{Z}\to \mbb{Z}/p^r\mbb{Z}$.  For convenience, we employ $\mbb{Z}/p^r\mbb{Z}$-gradings in this proof.  Under these new grading $\theta_r$ is an isomorphism onto the degree $0$ portion of $H^\ast(\D\G_{(r)},k)$.  For an element $a\in \mbb{Z}/p^r\mbb{Z}$ we let $\tilde{a}$ denote the unique representative of $a$ in $\{0,\dots, p^r-1\}$.
\par

Just as in Lemma~\ref{lem:HOdeg}, one can check that the natural map $\sigma_\D:\g^{(r)}\to H^2(\D\G_{(r)},k)$ has image in degree $p^r=0$ 
with respect to the $\mbb{Z}/p^r\mbb{Z}$-grading on cohomology.  We also have that the inflation $H^\ast(\G_{(r)},k)\to H^\ast(\D\G_{(r)},k)$ has image entirely in degree $0$, since the projection $\D\G_{(r)}\to k\G_{(r)}$ is graded with $k\G_{(r)}$ entirely in degree $0$.  By the same spectral sequence calculation as was given in the proof of Theorem~\ref{thm:|log|}, we find that
\[
\theta_r:H^\ast(\G_{(r)},k)\ot S(\g^{(r)})\to H^\ast(\D\G_{(r)},k)
\]
is an isomorphism onto the degree $0$ portion of the cohomology of $\D\G_{(r)}$.
\par

Under the induced $\mbb{Z}/p^r\mbb{Z}$-grading on the spectral sequence $\{E_s^{\ast,\ast}\}$ of the proof of Theorem~\ref{thm:|log|} we have
\[
(E^{i,j}_2)_{0}=H^i(\G_{(r)},k)\ot S^{j/2}(\g^{(r)})\ \ \mathrm{and}\ \ (E^{i,j}_2)_{a}=0\ \ \text{for each }a=\dim\G+1,\dots, p^r-1.
\]
This implies that $\left(H^\ast(\D\G_{(r)},k)\right)_a=0$ for each such $a$.  Hence, any homogenous element 
$\xi \in H^\ast(\D\bG_{(r)},k)$ of nonzero internal degree $\deg(\xi)$ satisfies  $\xi^m=0$, where
\[
m=\left\{\begin{array}{ll}
\left\lfloor\frac{p^r}{\widetilde{\deg(\xi)}}\right\rfloor & \text{if }\widetilde{\deg(\xi)}\nmid p^r\\\\
\frac{p^r}{\widetilde{\deg(\xi)}}-1 & \text{if }\widetilde{\deg(\xi)} \mid p^r,
\end{array}\right.
\]
since $\widetilde{\deg(\xi)}$ will be among $\dim\G+1,\dots, p^r-1$.  Said another way, the subspace $J$ spanned by elements of nonzero degree is contained in the nilradical, and the inclusion
\[
\left(H^{ev}(\D\G_{(r)},k)_0\right)_{\mathrm{red}}\to H^{ev}(\D\G_{(r)})_{\mathrm{red}}
\]
is therefore an isomorphism.  Since $\theta_r$ is an isomorphism onto the degree $0$ portion of cohomology, it follows that
\[
\theta^{ev}_{\mathrm{red}}:H^{ev}(\G_{(r)},k)_{\mathrm{red}}\ot S(\g^{(r)})\to H^{ev}(\D\G_{(r)},k)_{\mathrm{red}}
\]
is an isomorphism.  We take the spectrum to find that $\Theta_r$ is an isomorphism.
\end{proof}

\begin{theorem}\label{thm:|log|3}
Suppose $\G$ is a smooth algebraic group which admits a quasi-logarithm, and that $p>\dim\G+1$.  Then the image of $\theta_r$ in $H^\ast(\D\G_{(r)},k)$ has a complement $J$ which consists entirely of nilpotents, just as in Theorem~\ref{thm:|log|2}.  Furthermore, the map
\[
\Theta_r:|\D\G_{(r)}|\to |k\G_{(r)}|\times (\g^\ast)^{(r)}
\]
is an isomorphism for all $r$.
\end{theorem}

\begin{proof}
Our argument will be similar to that of Theorem~\ref{thm:|log|2}.  Via the projection $\mbb{Z}\to \mbb{Z}/p\mbb{Z}$ we get $\mbb{Z}/p\mbb{Z}$-gradings on the spectral sequence $\{E_r^{\ast,\ast}\}$ and the cohomology $H^\ast(\D\G_{(r)},k)$.  We have, under these $\mbb{Z}/p\mbb{Z}$-gradings, that
\[
(E^{i,j}_2)_{0}=H^i(\G_{(r)},k)\ot S^{j/2}(\g^{(r)}),\ \ (E^{i,j}_2)_{p-1}=(E^{i,j}_2)_{-1}=0,
\]
and that $\theta_r$ is an isomorphism onto the degree $0$ portion of cohomology $H^\ast(\D\G_{(r)},k)_0$.  Consider now any homogeneous element $\xi\in H^\ast(\D\G_{(r)},k)$ of degree $d\neq 0$.  Since $\mbb{Z}/p\mbb{Z}=\mbb{F}_p$ is a field there is a positive integer $d'\in \mbb{Z}$ which reduces to $-d^{-1}$ mod $p$.  We then find that $\xi^{d'}=0$, since
\[
\deg(\xi^{d'})=-d^{-1}d=-1\ \ \mathrm{and}\ \ H^\ast(\D\G_{(r)},k)_{-1}=0.
\]
Hence the subspace $J$ of element of nonzero degree is contained in the nilradical.  Just as before, this implies that $\theta_r$ induces an isomorphism
\[
\theta^{ev}_{\mathrm{red}}:H^{ev}(\G_{(r)},k)_{\mathrm{red}}\ot S(\g^{(r)})\to H^{ev}(\D\G_{(r)},k)_{\mathrm{red}},
\]
and that $\Theta_r$ is an isomorphism as well.
\end{proof}

One considers the examples of Section~\ref{sect:ql} to arrive at

\begin{corollary}\label{cor:|class|}
Let $\G$ be a general linear group, simple algebraic group, Borel subgroup in a simple algebraic group, or a unipotent subgroup in a semisimple algebraic group which is normalized by a maximal torus.  Suppose that $p$ is very good for $\G$, or that $p>cl(\G)$ in the unipotent case.
\begin{itemize}
\item If $p>\dim\G+1$ then $\Theta_r$ is an isomorphism for all $r$.
\item For arbitrary $p$ satisfying the hypothesis, the map $\Theta_r$ is an isomorphism whenever $r$ is such that $p^r>2\dim\G$.
\end{itemize}
\end{corollary}

\section{Results for support varieties}
\label{sect:supp}

For a Hopf algebra $A$ and finite dimensional $A$-module $M$ we let $|A|_M$ denote the support variety for $M$.  This is the closed, reduced, subscheme in $|A|$ defined by the kernel of the algebra map
\begin{equation}\label{eq:90}
-\ot M:H^{ev}(A,k)\to \Ext_A^{ev}(M,M).
\end{equation}
In this section we consider the support $|\D\G_{(r)}|_M$ associated to a finite dimensional $\D\G_{(r)}$-module $M$.  We show that there is a finite scheme map
\[
\Theta_r^M:|\D\G_{(r)}|_M\to |k\G_{(r)}|_M\times (\g^\ast)^{(r)}
\]
for any $M$ with trivial restriction to $\O(\G_{(r)})$, and that $\Theta_r^M$ is an isomorphism whenever $M$ is irreducible and $\G$ is a classical group at a large prime or large $r$.

\subsection{Generalities for support varieties}

Under the natural identification 
\[
\Ext_A^\ast(M,M)=\Ext^\ast_A(k,M\ot M^\ast)=H^\ast(A,M\ot M^\ast),
\]
\eqref{eq:90} corresponds to the mapping
\[
\mathrm{coev}_\ast^M:H^{ev}(A,k)\to H^{ev}(A,M\ot M^\ast)
\]
induced by the coevaluation $\mathrm{coev}^M:k\to M\ot M^\ast$~\cite[Prop.\ 2.10.8]{egno15}.  The algebra structure on $H^{ev}(A,M\ot M^\ast)$ is induced by the algebra structure on $M\ot M^\ast\cong \End_k(M,M)$.  By~\cite[Thm.\ VII.4.1]{maclane} (see also~\cite{suarez04}) the image of $H^{ev}(A,k)$ lies in the center of $H^{ev}(A,M\ot M^\ast)$.
\par

For $\G$ smooth, and any finite dimensional $\D\G_{(r)}$-module $M$, $\theta_r$ produces an algebra map
\begin{equation}\label{eq:fM}
f_{r,M}:H^{ev}(\G_{(r)},k)\ot S(\g^{(r)}[2])\to H^{ev}(\D\G_{(r)},M\ot M^\ast).
\end{equation}
Explicitly, $f_{r,M}$ is the composite
\[
H^{ev}(\G_{(r)},k)\ot S(\g^{(r)}[2])\overset{\theta_r}\longrightarrow H^{ev}(\D\G_{(r)},k)\overset{\mathrm{coev}^M_\ast}\longrightarrow H^{ev}(\D\G_{(r)},M\ot M^\ast).
\]
By the definition of $f_{r,M}$, one sees that the reduced subscheme in $|k\G_{(r)}|\times (\g^\ast)^{(r)}$ defined by the kernel of $f_{r,M}$ is exactly the image of $|\D\G_{(r)}|_M$ under $\Theta_r:|\D\G_{(r)}|\to |k\G_{(r)}|\times(\g^\ast)^{(r)}$.
\par

By the material of Section~\ref{sect:qlog} we understand that $\Theta_r$ is often an isomorphism.  However, by finiteness of $\Theta_r$ in general, we can adapt an argument of~\cite{friedlanderparshall87} in all circumstances to arrive at

\begin{proposition}[{\cite[Prop.\ 1.5]{friedlanderparshall87}}]
A finite dimensional $\D\G_{(r)}$-module $M$ is projective (or, equivalently, injective) as a  $\D\G_{(r)}$-module if and only if $\Theta_r\left(|\D\G_{(r)}|_M\right)=\{0\}$.
\end{proposition}

\begin{proof}
One simply repeats the proof of~\cite[Prop.\ 1.5]{friedlanderparshall87}, using the fact that $\rep(\D\G_{(r)})$ is a Frobenius category~\cite{larsonsweedler69}.
\end{proof}

For the remainder of the section we seek to give a more precise description of the support $|\D\G_{(r)}|_M$ for a finite dimensional $\D\G_{(r)}$-module
$M$ whose restriction to $\O(\G_{(r)})$ is trivial (and thus arises as the restriction along the the natural quotient 
$\D\G_{(r)} \twoheadrightarrow kG_{(r)}$ of a $kG_{(r)}$module  which we also denote by $M$).  By Proposition \ref{prop:simpmod}, this condition is satisfied by any irreducible $\D\G_{(r)}$-module.  Whenever $M$ satisfies this condtion, there is a natural inflation map $H^*(\G_{(r)},M) \to H^*(\D\G_{(r)},M)$.

In the statement of the following lemma, we consider the algebra map
\[
\theta_{r,M}:H^{ev}(\G_{(r)},M\ot M^\ast)\ot S(\g^{(r)}[2])\to H^{ev}(\D\G_{(r)},M\ot M^\ast)
\]
induced by the inflation $H^{ev}(\G_{(r)},M\ot M^\ast)\to H^{ev}(\D\G_{(r)},M\ot M^\ast)$ and the map from $S(\g^{(r)}[2])$ defined via $\sigma'_\D$ as above.

\begin{lemma}\label{lem:suppsq}
For any finite dimensional $\D\G_{(r)}$-module $M$ whose restriction to $\O(\G_{(r)})$ is trivial, the following diagram commutes
\[
\xymatrix{
H^{ev}(\G_{(r)},k)\ot S(\g^{(r)}[2])\ar[d]_{\theta_r}\ar@{-->}[drr]|{f_{r,M}}\ar[rr]^(.45){\mathrm{coev}^M_\ast\ot id_S} & & H^{ev}(\G_{(r)},M\ot M^\ast)\ot S(\g^{(r)}[2])\ar[d]^{\theta_{r,M}}\\
H^{ev}(\D\G_{(r)},k)\ar[rr]_(.45){\mathrm{coev}^M_\ast} & & H^{ev}(\D\G_{(r)},M\ot M^\ast).
}
\]
\end{lemma}

\begin{proof}
It suffices to prove that the two maps
\[
H^{ev}(\G_{(r)},k)\ot S(\g^{(r)}[2])\rightrightarrows H^{ev}(\D\G_{(r)},M\ot M^\ast)
\]
agree on the factors $H^{ev}(\G_{(r)},k)$ and $S(\g^{(r)}[2])$ independently.  The two restrictions to $S(\g^{(r)}[2])$ are equal, since they are both defined as the composition
\[
S(\g^{(r)}[2])\overset{\sigma'_\D}\longrightarrow H^\ast(\D\G_{(r)},k)\overset{\mathrm{coev}^M_\ast}\longrightarrow H^\ast(\D\G_{(r)},M\ot M^\ast). 
\]
So we need only establish commutativity of the diagram
\[
\xymatrix{
H^{ev}(\G_{(r)},k)\ar[d]_{\mathrm{res}}\ar[r]^(.4){\mathrm{coev}^M_\ast} & H^{ev}(\G_{(r)},M\ot M^\ast)\ar[d]^{\mathrm{res}}\\
H^{ev}(\D\G_{(r)},k)\ar[r]^(.42){\mathrm{coev}^M_\ast} & H^{ev}(\D\G_{(r)},M\ot M^\ast),
}
\]
which follows by functoriality of the inflation map.
\end{proof}

\begin{proposition}\label{prop:contain}
For any finite dimensional $\D\G_{(r)}$-module $M$ whose restriction to $\O(\G_{(r)})$ is trivial (for example, if $M$ is irreducible), the restriction of $\Theta_r: \D\G_{(r)}|_M\to |k\G_{(r)}|\times (\g^\ast)^{(r)}$ to $:|\D\G_{(r)}|_M$ factors through the closed subscheme $|k\G_{(r)}|_M\times (\g^\ast)^{(r)}$,
determining a finite map of schemes \[
\Theta_{r,M}:|\D\G_{(r)}|_M\to |k\G_{(r)}|_M\times (\g^\ast)^{(r)}.
\]
\end{proposition}

\begin{proof}
The image of $\Theta_r|_{|\D\G_{(r)}|_M}$ is the closed subscheme defined by the kernel of $f_{r,M}$.  By Lemma~\ref{lem:suppsq},
$f_{r,M}$ factors through the product map
\[
\mathrm{coev}^M_\ast\ot id_S:H^{ev}(\G_{(r)},k)\ot S(\g^{(r)}[2])\to H^{ev}(\G_{(r)},M\ot M^\ast)\ot S(\g^{(r)}[2]),
\]
and hence
\[
\mathrm{ker}(\mathrm{coev}^M_\ast)\ot S(\g^{(r)}[2])\subset \ker(f_{r,M}).
\]
It follows that $\Theta_r|_{|\D\G_{(r)}|_M}$ factors through $|k\G_{(r)}|_M\times (\g^\ast)^{(r)}$.
\end{proof}

\subsection{Support varieties for classical groups}

We now consider irreducible modules and classical groups.  We fix $\mbb{G}$ a smooth algebraic group.

\begin{lemma}\label{lem:0Minj}
Suppose $\G$ admits a quasi-logarithm and that $V$ is an irreducible $\D\G_{(r)}$-module.  Suppose additionally that $p^r>\dim\G$.  Then the map
\[
\theta_{r,V}:H^{ev}(\G_{(r)},V\ot V^\ast)\ot S(\g^{(r)}[2])\to H^{ev}(\D\G_{(r)},V\ot V^\ast)
\]
is injective.
\end{lemma}

\begin{proof}
Take $\O=\O(\G_{(r)})$.  It suffices to show that the associated graded map $\mathrm{gr}\ \!\theta_{r,V}$ is injective, under some filtration.
\par

We consider the Grothendieck spectral sequence
\[
E_2^{i,j}=H^i(\G_{(r)},H^j(\O,V\ot V^\ast))\Rightarrow H^{i+j}(\D\G_{(r)},V\ot V^\ast)
\]
induced by the sequence $1\to \O\to \D\G_{(r)}\to k\G_{(r)}\to 1$.  Recall, from Proposition~\ref{prop:simpmod}, that the $\O$ acts trivially on $V$ and $V^\ast$.  Whence we may rewrite the above spectral sequence as
\[
E_2^{i,j}=H^i(\G_{(r)},(\wedge^{j_1}\g)\ot V\ot V^\ast)\ot S^{j_2}(\g^{(r)}[2])\Rightarrow H^{i+j}(\D\G_{(r)},V\ot V^\ast).
\]
Since $\O$ acts trivially on $V$ and $V^\ast$, the $\D\G_{(r)}$-module $V\ot V^\ast$ is graded and concentrated in degree $0$, under the $\mathbb{Z}$-grading on $\D\G_{(r)}$ induced by any quasi-logarithm on $\G$.  Now one can argue just as in the proof of Theorem~\ref{thm:|log|}, using the grading on the above spectral sequence induced by the quasi-logarithm, to conclude that $\theta_{r,V}$ is injective.
\end{proof}

\begin{theorem}\label{thm:|log|V}
Suppose $\G$ admits a quasi-logarithm and that $V$ is an irreducible $\D\G_{(r)}$-module.  Then the scheme map
\[
\Theta_{r,V}:|\D\G_{(r)}|_V\to |k\G_{(r)}|_V\times (\g^\ast)^{(r)}
\]
is finite and surjective.  Furthermore, when $p>\dim\G+1$ or $p^r>2\dim \G$ the map $\Theta_{r,V}$ is an isomorphism.
\end{theorem}

\begin{proof}
Finiteness follows by finiteness of $\Theta_r$.  So we need only check surjectivity to establish the first claim.  We omit the shift $[2]$ in the symmetric algebra to ease notation.  As discussed above, the image of $\Theta_{r,V}$ is the subscheme associated to the kernel of the algebra map
\[
f_{r,V}:H^{ev}(\G_{(r)},k)\ot S(\g^{(r)})\to H^{ev}(\D\G_{(r)},V\ot V^\ast),
\]
which was defined at~\eqref{eq:fM}.  Now, by Lemma~\ref{lem:suppsq}, we have that $f_{r,V}$ factors as the composite of
\[
\mathrm{coev}^V_\ast\ot id_S:H^{ev}(\G_{(r)},k)\ot S(\g^{(r)})\to H^{ev}(\G_{(r)},V\ot V^\ast)\ot S(\g^{(r)})
\]
with
\[
\theta_{r,V}:H^{ev}(\G_{(r)},V\ot V^\ast)\ot S(\g^{(r)})\to H^{ev}(\D\G_{(r)},V\ot V^\ast).
\]
By Lemma~\ref{lem:0Minj}, $\theta_{r,V}$ is injective.  Hence it follows that $\ker(f_{r,V})=\ker(\mathrm{coev}^V_\ast)\ot S(\g^{(r)})$ and subsequently
\[
\Theta_{r,V}\left(|\D\G_{(r)}|_V\right)=|k\G_{(r)}|_V\times (\g^\ast)^{(r)}.
\]
The fact that $\Theta_{r,V}$ is an isomorphism when $p>\dim\G+1$ or $p^r>2\dim\G$ follows from the fact that $\Theta_r$ is an isomorphism in these cases, by Theorems~\ref{thm:|log|2} and~\ref{thm:|log|3}.
\end{proof}

We apply the theorem in the classical settings to find

\begin{corollary}\label{cor:|class|V}
Let $\G$ be a general linear group, simple algebraic group, Borel subgroup in an simple algebraic group, or a unipotent subgroup in a semisimple algebraic group which is normalized by a maximal torus.  Suppose that $p$ is very good for $\G$, or that $p>cl(\G)$ in the unipotent case.
\begin{itemize}
\item If $p>\dim\G+1$ then $\Theta_{r,V}$ is an isomorphism for every $r$ and irreducible $\D\G_{(r)}$-module $V$.
\item For arbitrary $p$ satisfying the hypothesis, and $r$ such that $p^r>2\dim\G$, the map $\Theta_{r,V}$ is an isomorphism for every irreducible $\D\G_{(r)}$-module $V$.
\end{itemize}
\end{corollary}

\begin{proof}
By Propositions~\ref{prop:ql1} and~\ref{prop:ql3}, and Corollary~\ref{cor:ql2}, the group $\G$ will admit a quasi-logarithm.  Hence we may apply Theorem~\ref{thm:|log|V}.
\end{proof}

\bibliographystyle{abbrv}

\begin{thebibliography}{10}

\bibitem{AS}
N.~Andruskiewitsch and H.-J. Schneider.
\newblock Pointed {H}opf algebras.
\newblock {\em New directions in Hopf algebras}, 43:1--68, 2002.

\bibitem{atv07}
M.~Artin, J.~Tate, and M.~Van~den Bergh.
\newblock Some algebras associated to automorphisms of elliptic curves.
\newblock {\em The Grothendieck Festschrift}, pages 33--85, 2007.

\bibitem{bnpp14}
C.~Bendel, D.~Nakano, B.~Parshall, and C.~Pillen.
\newblock {\em Cohomology for quantum groups via the geometry of the nullcone},
  volume 229.
\newblock American Mathematical Society, 2014.

\bibitem{benkartetal10}
G.~Benkart, M.~Pereira, and S.~Witherspoon.
\newblock Yetter--{D}rinfeld modules under cocycle twists.
\newblock {\em J. Algebra}, 324(11):2990--3006, 2010.

\bibitem{bkv16}
R.~Bezrukavnikov, D.~Kazhdan, and Y.~Varshavsky.
\newblock On the depth $r$ {B}ernstein projector.
\newblock {\em Selecta Math.}, 22(4):2271--2311, 2016.

\bibitem{cline87}
E.~Cline.
\newblock Simulating algebraic geometry with algebra {III}: {T}he {L}usztig
  conjecture as a {$TG_1$}-problem.
\newblock In {\em AMS Proc. Symp. Pure Math}, volume~47, pages 149--161, 1987.

\bibitem{drupieski16}
C.~Drupieski.
\newblock Cohomological finite-generation for finite supergroup schemes.
\newblock {\em Adv. Math}, 288:1360--1432, 2016.

\bibitem{etingof02}
P.~Etingof.
\newblock On {V}afa's theorem for tensor categories.
\newblock {\em Math. Res. Lett.}, 9(5-6):651--657, 2002.

\bibitem{etingofgelaki02}
P.~Etingof and S.~Gelaki.
\newblock On the quasi-exponent of finite-dimensional {H}opf algebras.
\newblock {\em Math. Res. Lett.}, 9(2-3):277--287, 2002.

\bibitem{etingofnikshychostrik11}
P.~Etingof, D.~Nikshych, and V.~Ostrik.
\newblock Weakly group-theoretical and solvable fusion categories.
\newblock {\em Adv. Math.}, 226(1):176--205, 2011.

\bibitem{etingofostrik04}
P.~Etingof and V.~Ostrik.
\newblock Finite tensor categories.
\newblock {\em Mosc. Math. J}, 4(3):627--654, 2004.

\bibitem{egno15}
P.~I. Etingof, S.~Gelaki, D.~Nikshych, V.~Ostrik, S.~Gelaki, and D.~Nikshych.
\newblock {\em Tensor categories}, volume 205.
\newblock American Mathematical Society, 2015.

\bibitem{friedlanderparshall86}
E.~M. Friedlander and B.~J. Parshall.
\newblock Cohomology of {L}ie algebras and algebraic groups.
\newblock {\em Amer. J. Math.}, 108(1):235--253, 1986.

\bibitem{friedlanderparshall87}
E.~M. Friedlander and B.~J. Parshall.
\newblock Geometry of $p$-unipotent {L}ie algebras.
\newblock {\em J. Algebra}, 109(1):25--45, 1987.

\bibitem{friedlandersuslin97}
E.~M. Friedlander and A.~Suslin.
\newblock Cohomology of finite group schemes over a field.
\newblock {\em Invent. Math.}, 127(2):209--270, 1997.

\bibitem{garibaldi09}
S.~Garibaldi.
\newblock Vanishing of trace forms in low characteristics.
\newblock {\em Algebra Number Theory}, 3(5):543--566, 2009.

\bibitem{gelaki12}
S.~Gelaki.
\newblock Module categories over affine group schemes.
\newblock {\em Quantum Topol.}, 6(1):1--37, 2015.

\bibitem{gelakinaidunikshych09}
S.~Gelaki, D.~Naidu, and D.~Nikshych.
\newblock Centers of graded fusion categories.
\newblock {\em Algebra Number Theory}, 3(8), 2009.

\bibitem{gerstenhaber64}
M.~Gerstenhaber.
\newblock On the deformation of rings and algebras.
\newblock {\em Ann. of Math.}, pages 59--103, 1964.

\bibitem{ginzburgkumar93}
V.~Ginzburg and S.~Kumar.
\newblock Cohomology of quantum groups at roots of unity.
\newblock {\em Duke Math. J}, 69(1):179--198, 1993.

\bibitem{grothendieck57}
A.~Grothendieck.
\newblock Sur quelques points d'alg\`{e}bre homologique.
\newblock {\em Tohoku Math. J. (2)}, 9(2):119--183, 1957.

\bibitem{jantzen07}
J.~C. Jantzen.
\newblock {\em Representations of algebraic groups}.
\newblock Number 107. American Mathematical Soc., 2007.

\bibitem{ksz06}
Y.~Kashina, Y.~Sommerh{\"a}user, and Y.~Zhu.
\newblock On higher {F}robenius-{S}chur indicators.
\newblock {\em Mem. Amer. Math. Soc.}, 181(855):viii+65, 2006.

\bibitem{kazhdanvarshavsky06}
D.~Kazhdan and Y.~Varshavsky.
\newblock Endoscopic decomposition of certain depth zero representations.
\newblock {\em Progress in Mathematics-Boston-}, 243:223, 2006.

\bibitem{larsonsweedler69}
R.~G. Larson and M.~E. Sweedler.
\newblock An associative orthogonal bilinear form for {H}opf algebras.
\newblock {\em Amer. J. Math.}, 91(1):75--94, 1969.

\bibitem{maclane}
S.~MacLane.
\newblock {\em Homology}.
\newblock Springer Science \& Business Media, 2012.

\bibitem{majidoeckl99}
S.~Majid and R.~Oeckl.
\newblock Twisting of quantum differentials and the planck scale hopf algebra.
\newblock {\em Comm. Math. Phys.}, 205(3):617--655, 1999.

\bibitem{massey54}
W.~S. Massey.
\newblock Products in exact couples.
\newblock {\em Ann. of Math.}, pages 558--569, 1954.

\bibitem{mpsw10}
M.~Mastnak, J.~Pevtsova, P.~Schauenburg, and S.~Witherspoon.
\newblock Cohomology of finite-dimensional pointed hopf algebras.
\newblock {\em Proc. Lond. Math. Soc.}, 100(2):377--404, 2010.

\bibitem{matsumura}
H.~Matsumura.
\newblock {\em Commutative ring theory}, volume~8.
\newblock Cambridge university press, 1989.

\bibitem{montgomery93}
S.~Montgomery.
\newblock {\em Hopf algebras and their actions on rings}.
\newblock Number~82. American Mathematical Soc., 1993.

\bibitem{montgomery04}
S.~Montgomery.
\newblock Algebra properties invariant under twisting.
\newblock {\em Hopf Algebras in Noncommutative Geometry and Physics. Lecture
  Notes in Pure and Appl. Math}, 239:229--243, 2004.

\bibitem{ngschauenburg07}
S.-H. Ng and P.~Schauenburg.
\newblock Frobenius-{S}chur indicators and exponents of spherical categories.
\newblock {\em Adv. Math.}, 211(1):34--71, 2007.

\bibitem{pevtsova}
J.~Pevtsova.
\newblock Representations and cohomology of finite group schemes.
\newblock In {\em Advances in representation theory of algebras}, EMS Ser.
  Congr. Rep., pages 231--261. Eur. Math. Soc., Z{\"u}rich, 2013.

\bibitem{pevtsovawitherspoon09}
J.~Pevtsova and S.~Witherspoon.
\newblock Varieties for modules of quantum elementary abelian groups.
\newblock {\em Algebr. Represent. Theory}, 12(6):567--595, 2009.

\bibitem{radford93}
D.~E. Radford.
\newblock Minimal quasitriangular {H}opf algebras.
\newblock {\em J. Algebra}, 157(2):285--315, 1993.

\bibitem{seitz00}
G.~M. Seitz.
\newblock Unipotent elements, tilting modules, and saturation.
\newblock {\em Invent. Math.}, 141(3):467--502, 2000.

\bibitem{shimizu17}
K.~Shimizu.
\newblock Integrals for finite tensor categories.
\newblock {\em preprint \href{https://arxiv.org/abs/1702.02425}{\tt
  arxiv.org/abs/1702.02425}}.

\bibitem{stefanvay16}
D.~{\c{S}}tefan and C.~Vay.
\newblock The cohomology ring of the 12-dimensional {F}omin--{K}irillov
  algebra.
\newblock {\em Adv. Math.}, 291:584--620, 2016.

\bibitem{suarez04}
M.~Suarez-Alvarez.
\newblock The {H}ilton-{H}eckmann argument for the anti-commutativity of cup
  products.
\newblock {\em Proc. Amer. Math. Soc.}, 132(8):2241--2246, 2004.

\bibitem{sullivan78}
J.~B. Sullivan.
\newblock Representations of the hyperalgebra of an algebraic group.
\newblock {\em Amer. J. Math.}, 100(3):643--652, 1978.

\bibitem{waterhouse12}
W.~C. Waterhouse.
\newblock {\em Introduction to affine group schemes}, volume~66.
\newblock Springer Science \& Business Media, 2012.

\bibitem{weibel95}
C.~A. Weibel.
\newblock {\em An introduction to homological algebra}.
\newblock Number~38. Cambridge university press, 1995.

\end{thebibliography}

\end{document}